\documentclass{article}

\usepackage[english]{babel}
\usepackage[utf8]{inputenc}
\usepackage[T1]{fontenc}

\usepackage{mathptmx}
\DeclareMathSymbol\varGamma    {\mathord}{letters}{0}
\DeclareMathSymbol\varOmega    {\mathord}{letters}{10}
\renewcommand{\Lambda}{{\mit\varLambda}}
\renewcommand{\Omega}{{\mit\varOmega}}

\usepackage[sort&compress]{natbib}

\usepackage{amsmath} 
\usepackage{amssymb}

\usepackage{amsthm} 
\usepackage{mathtools}
\usepackage{enumerate} 

\usepackage{geometry}
\usepackage{authblk}

\renewcommand{\leq}{\leqslant}
\renewcommand{\geq}{\geqslant}

\newcommand{\eps}{\varepsilon}
\newcommand{\mf}{\mathbf}

\newcommand{\e}{+}
\newcommand{\I}{-}

\DeclareMathOperator{\myRe}{\mathrm{Re}}
\DeclareMathOperator{\sech}{\mathrm{sech}}

\newcommand{\R}{\mathbb{R}}
\newcommand{\C}{\mathbb{C}}
\newcommand{\Z}{\mathbb{Z}}
\newcommand{\N}{\mathbb{N}}

\newcommand{\diff}{\,\mathrm{d}}

\theoremstyle{plain}

\newtheorem{prop}{Proposition}
\newtheorem{thm}{Theorem}
\newtheorem{lemme}{Lemma}

\theoremstyle{remark}

\newtheorem{rmk}{Remark}

\title{Regularity results for transmission problems with sign-changing coefficients: a modal approach}

\author{Valentin Vinoles%
  \thanks{\texttt{valentin.vinoles@epfl.ch}}}
\affil{\'Ecole Polytechnique F\'ed\'erale de Lausanne \\ SB MATHAA CAMA, Station 8 \\ CH-1015 Lausanne, Switzerland}

\begin{document}

\maketitle

\begin{abstract}
 We investigate some scalar transmission problems between a classical positive material and a negative one, whose physical coefficients are negative. First, we consider cases where the negative inclusion is a disk in 2d and a ball in 3d. Thanks to asymptotics of Bessel functions (validated numerically), we show well-posedness but with some possible loses of regularity of the solution compared to the classical case of transmission problems between two positive materials. Noticing that the curvature plays a central role, we then explore the case of flat interfaces in the context of waveguides. In this case, the transmission problem can also have some loses of regularity, or even be ill-posed (kernel of infinite dimension). \newline
 

\noindent 2000 Math Subject Classification: 33C10, 35A01, 35B65, 35J05
\end{abstract}

\section{Introduction}

In recent decades, physicists and engineers have studied and developed metamaterials, \textit{i.e.} artificial materials with unusual electromagnetic properties through periodic microscopic structures that resonate. In particular, some of them exhibit effective permittivity $\eps$ and/or permeability $\mu$ that are negative in certain ranges of frequencies (see \cite{pendry2004negative,smith2004metamaterials}, the mathematical justification of these effective behaviours is based on the so-called high contrast homogenization, see for instance \cite{bouchitte2010homogenization,lamacz2013effective}). Such media are subject of intense researches due to promising applications (\cite{cui2010metamaterials}): super-lens, cloaking, improved antenna, etc. 

In this paper, we study scalar transmission problems in the frequency domain between a \emph{positive material}, that is to say a medium with both positive $\eps$ and $\mu$, and a \emph{negative material}, with both negative $\eps$ and $\mu$. Since the permittivity and permeability change sign through the interface between the two materials, we refer to these problems as (scalar) transmission problems with sign-changing coefficients. 

From a mathematical point of view, they raise new and interesting questions that require specific tools. There is already a relatively abundant mathematical literature on these problems. We refer to the survey \cite{li2016literature}. Without being exhaustive, let us mention first the pioneering work of \cite{costabel1985direct} using boundary integral techniques, followed by \cite{ola1995remarks}. There are also the works around the cloaking and the so-called ``anomalous localized resonances'' (\textit{e.g.} \cite{bouchitte2010cloaking,milton2006cloaking}). Another important contribution is the series of papers by Bonnet-Ben Dhia \textit{et al.} using the T-coercivity method (see for instance \cite{dhia2012t,bonnet2013radiation} and references therein). An alternative method is the reflecting technique introduced by \cite{nguyen2015asymptotic,nguyen2016limiting}. Finally, some authors studied the links between these problems and some transmission problems in the time domain, based on the limiting amplitude principle (\cite{cassier2014etude,gralak2012negative}).

It is nowadays well-known that well-posedness of transmission problems with sign-changing coefficients is related to the \emph{contrasts}, defined as the ratios of the values of the coefficients on each side of the interface between the two materials. In order to ensure well-posedness in the classical $H^1$ framework, the contrast in the principal part of the operator (for scalar problems)  must lie outside an interval called the critical interval that contains $\{-1\}$ (see \textit{e.g.} \cite{dhia2012t}). If the interface is smooth, this interval reduces to $\{-1\}$.

The critical case of contrasts equal to $-1$ has not been much studied. As pointed out by \cite{ola1995remarks} and \cite{nguyen2016limiting}, this case can lead to loses of regularity of the solutions in the sense that they are less regular than in the classical case (\textit{i.e.} transmission problems between two positive materials). In this paper, we want to investigate more on these loses of regularity. We restrain ourselves to particular geometries for which one can use modal decomposition techniques based on the separation of variables (\cite{morse1953methods}). While these methods are well-known and widely used, their application to transmission problems with sign-changing coefficients are not without interest: they manage to fully describe the well-posedness of our problems and the regularity of the solutions, and gives optimal results. Incidentally, as we will see, mention that they are well adapted to the description of the radiation conditions which can be tricky in negative materials (\cite{malyuzhinets1951note,ziolkowski2001wave}). Let us mention that the present paper is a revised version of the study presented in the PhD thesis of the author (\cite{vinoles2016problemes}).

This text is organized as follows. First, in Section \ref{sec:mathpb} we settle the transmission problems with sign-changing coefficients we study. In Section \ref{sec:diskball}, we explore the case where the negative material is a disk (in 2d) and a ball (in 3d). The analysis requires estimates for Bessel and Hankel functions that are not totally standard (proved in the Appendix and verified numerically). We also study what happens when the curvature tends to 0. That motivates the study of Section \ref{sec:planar} in which we investigate some cases where the interface is flat. We conclude in Section \ref{ref:comments} with some comments and perspectives.

\section{Setting of the problem and objectives}
\label{sec:mathpb}

Consider for the moment a generic non-empty simply connected open set $\Omega^\I$ (not necessarily bounded) of $\R^d$ where $d \geq 1$ is the dimension and define its exterior $\Omega^\e$ by
\begin{equation}
\Omega^\e := \R^d \setminus \overline{\Omega^\I}.
\end{equation}
We assume that $\Omega^+ \neq \emptyset$ and that the interface $\Gamma$ between $\Omega^\e$ and $\Omega^\I$ is smooth. We denote by $\mf n$ the outward-pointing normal of $\partial \Omega^\I$ (see Figure \ref{fig:geometry}).
\begin{figure}[!htbp]  
\centering
\includegraphics{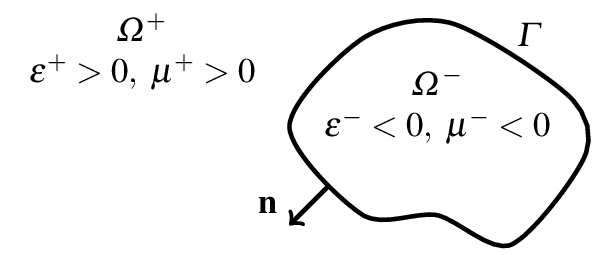}
\caption{\label{fig:geometry}Geometry of the problem (in the case $\Omega^\I$ bounded).}
\end{figure}
The domain $\Omega^\e$ (resp. $\Omega^\I$) represents the positive (resp. negative) material. More precisely, we consider two functions $\eps$ and $\mu$, representing for instance the permittivity and the permeability, such that
\begin{equation}
\eps(\mf x) := \begin{cases} \eps^\e > 0 & \text{for $\mf x \in\Omega^\e$,} \\ \eps^\I < 0 & \text{for $\mf x \in\Omega^\I$,}  \end{cases} \qquad \text{and} \qquad 
\mu(\mf x) := \begin{cases} \mu^\e > 0 & \text{for $\mf x \in\Omega^\e$,} \\ \mu^\I < 0 & \text{for $\mf x \in\Omega^\I$,} \end{cases}
\end{equation}
where $\eps^\e$ and $\mu^\e$ (resp. $\eps^\I$ and $\mu^\I$) are positive (resp. negative) constants. We define $\omega > 0$ the frequency with the convention of the time dependence $e^{-i \omega t}$. Introduce $k^\e$ and $k^\I$ the wave numbers such that
\begin{equation}
 k^\e := \omega \sqrt{\eps^\e \mu^\e}  \qquad \text{and} \qquad k^\I := \omega \sqrt{\eps^\I \mu^\I}.
\end{equation}
Notice that these quantities are positive.

Let us now introduce the couple $(\sigma,\varsigma)$ equal indifferently to $(\eps^{-1},\mu)$ or $(\mu^{-1},\eps)$. We look for a solution $u$ of the Helmholtz equation 
\begin{equation}\label{eq:TE}
 \nabla \cdot \left( \sigma(\mf x) \nabla u(\mf x) \right) + \omega^2 \varsigma(\mf x) u(\mf x) =0,\qquad \mf x \in \R^d,
\end{equation}
that can be written as
\begin{equation} \label{eq:TE2}
\left\{
\begin{aligned}
 & \Delta u + (k^\pm)^2 u = 0, &  \qquad \text{in $\Omega^\pm$,} \\
[u]_\Gamma &= 0,\ [ \sigma \partial_{\mf n} u ]_\Gamma = 0,
\end{aligned}
\right.
\end{equation}
where $[\,\cdot\,]_\Gamma$ stands for the jump through $\Gamma$ from $\Omega^\e$ to $\Omega^\I$ (here $\partial_{\mf n} u  := \nabla u \cdot\, \mf n$). Notice that the change of sign in \eqref{eq:TE2} only appears in the jump of the fluxes through $\Gamma$ and not in equations in $\Omega^\e$ and $\Omega^\I$.  

Consider now an incident field $u^\mathrm{inc}$ that satisfies
\begin{equation}\label{eq:pbondeinc} 
\Delta u^\mathrm{inc} + (k^\e)^2 u^\mathrm{inc}= 0,  \qquad \text{in $\Omega^\e$,} 
\end{equation}
and split $u$ as, on the one hand, the sum of the incident wave $u^\mathrm{inc}$ and a scattered one denoted by $u^\e$ in $\Omega^\e$ and, on the other hand, as a transmitted wave denoted by $u^\I$ in $\Omega^\I$ :
\begin{equation}\label{eq:decomposition} 
u(\mf x) = \begin{cases} u^\mathrm{inc}(\mf x) + u^\e(\mf x) & \text{for $\mf x \in \Omega^\e$,} \\ u^\I(\mf x) & \text{for $\mf x \in \Omega^\I$.} \end{cases}
\end{equation}
Using \eqref{eq:decomposition}, the transmission conditions of \eqref{eq:TE2} write
\begin{equation}
\left. u^\I \right \vert_\Gamma - \left. u^\e \right \vert_\Gamma =  f \qquad \text{and} \qquad
 \partial_{\mf n} u^\I- \kappa \partial_{\mf n}  u^\e  = g, 
\end{equation}
where $f:= \left. u^\mathrm{inc}  \right \vert_\Gamma$, $g := \kappa \partial_{\mf n}  u^\mathrm{inc}$ and
where $\kappa$ is the contrast defined by
\begin{equation}
\kappa := \frac{\sigma^\e}{\sigma^\I}  < 0.
\end{equation}  
In order to close \eqref{eq:TE2}, one needs to add Radiation Conditions (RCs) when $\vert \mf x \vert$ tends to $+\infty$. There are two possible cases. When $\Omega^\I$ is bounded, one must impose the RCs in the positive medium $\Omega^\e$ only. In this case, one classically uses the Sommerfeld radiation condition (\cite{cakoni2005qualitative,colton2012inverse}): 
\begin{equation}\label{eq:sommerfeld}
\lim_{R \to + \infty} \int_{\vert \mf x \vert = R} \left\vert \frac{\partial }{\partial r}u - i k^\e u \right \vert^2 \diff r =0.
\end{equation}
The other case where $\Omega^\I$ is unbounded is less classical. This case is handled later on in Section \ref{sec:The unbounded case}.  

At the end, we obtain the following transmission problem:
\begin{equation} \label{eq:pbtrans} 
\left\{
\begin{aligned}
 \Delta u^\pm + (k^\pm)^2 u^\pm &= 0,&  \qquad &\text{in $\Omega^\pm$,} \\
u^\I -u^\e &=  f, & \qquad  &\text{on $\Gamma$,}  \\
\partial_{\mf n}  u^\I  - \kappa \partial_{\mf n}  u^\e &= g, & \qquad &\text{on $\Gamma$,}  \\
+\ \text{RCs}, &  & \qquad & \text{when $\vert \mf x \vert \to +\infty$}.
\end{aligned}
\right.
\end{equation}

Let $(f,g) \in \mf H^s(\Gamma)$ be the given data of \eqref{eq:pbtrans}, where we define
\begin{equation}
\mf H^s(\Gamma) := H^s(\Gamma) \times H^{s-1}(\Gamma),\qquad s > 0,
\end{equation}
Notice that this space is ``natural''  as the data come from the trace and the normal trace of the incident field $u^\mathrm{inc}$ on $\Gamma$.

When $\Omega^\I$ is bounded and $\kappa \neq -1$, this is well-known that for all $(f,g) \in \mf H^s(\Gamma)$, $s>0$, \eqref{eq:pbtrans} admits a unique solution $(u^\I\vert_\Gamma,u^\e\vert_\Gamma) \in H^s(\Gamma)^2$ (see for instance \cite{costabel1985direct}). In other words, there is no regularity loss. In dimension $d \geq 3$, when $\kappa = -1$ and $\Omega^\I$ is bounded and strictly convex, \cite{ola1995remarks} (and \cite{nguyen2016limiting} later on in  a more general setting) proved that for all $(f,g) \in \mf H^{s}(\Gamma)$, $s>1$, \eqref{eq:pbtrans} admits a unique solution $(u^\I\vert_\Gamma,u^\e\vert_\Gamma) \in H^{s-1}(\Gamma)^2$ (one order of regularity lost).

Studying \eqref{eq:pbtrans} for general domains $\Omega^\e$ and $\Omega^\I$ appears to be difficult. In this paper, we shall use a more modest approach using modal decompositions, also called (generalized) Lorentz-Mie method in the physics/engineer communities. Recall (see \cite{cakoni2005qualitative,colton2012inverse,morse1953methods,taflove2005computational}) that it is based on the separation of variables that allows to reduce \eqref{eq:pbtrans} to a countable family of linear systems. The solvability of \eqref{eq:pbtrans} boils down to the solvability of all these systems and the regularity of the solutions is linked to the asymptotics of their modal coefficients. Let us also mention that the radiation conditions are easily handled by modal decompositions, one just needs to select the modes that satisfy such conditions. Finally, this method gives optimal results, in the sense that it gives the best regularity of the solution for a given regularity of the data.

In the next section (Section \ref{sec:diskball}), we deal with cases where $\Omega^\I$ is a disk (in 2d) and a ball (in 3d). For $d=3$, we recover the results of \cite{ola1995remarks} and \cite{nguyen2016limiting} and gives new results for $d=2$. In particular, this later case leads to larger loses of regularity. We also study what happens when the curvature tends to 0, that is to say when the radius tends to infinity. In Section \ref{sec:planar}, cases with unbounded $\Omega^\I$ and flat interfaces are explored.

As we will see, three situations can be encountered:
\begin{itemize}
\item \emph{the standard case $\kappa \neq -1$} (corresponding to $\sigma^\I / \sigma^\e \neq -1$ and $\varsigma^\I / \varsigma^\e \neq -1$). Here, nothing unusual happens and we recover the standard result of no regularity loss;
\item \emph{the critical case $\kappa = -1$ and $k^\e \neq k^\I$} (corresponding to $\sigma^\I / \sigma^\e = -1$ but $\varsigma^\I / \varsigma^\e \neq -1$). In this case, although \eqref{eq:pbtrans} can be uniquely solved, we can have some regularity losses.
\item \emph{the super-critical case $\kappa = -1$ and $k^\e = k^\I$} (corresponding to $\sigma^\I / \sigma^\e = \varsigma^\I / \varsigma^\e = -1$)  Here, the regularity losses are at least as important as the ones in the critical case. In some situations, \eqref{eq:pbtrans} can even be ill-posed. 
\end{itemize}

\begin{rmk}
In the following we focus on the regularity of the traces $(u^\I\vert_\Gamma,u^\e\vert_\Gamma)$. Indeed, since the change of sign of \eqref{eq:pbtrans} only appears in the transmission conditions and not in the volume equations, one has the standard regularity result: for $(u^\I\vert_\Gamma,u^\e\vert_\Gamma) \in H^s(\Gamma)^2$, one gets $(u^\I,u^\e) \in  H^{s+1/2}(\Omega^\I) \times H^{s+1/2}(\Omega^\e)$.
\end{rmk}

\section{The case where the negative material is a ball or a disk} \label{sec:diskball}
 
We consider now that $\Omega^\I$ is a disk (in 2d) or a ball (in 3d) of radius $R>0$ centred at the origin (see Figure \ref{fig:DiskBall}). As mentioned before, since $\Omega^\I$ is bounded, the radiation condition is simply the Sommerfeld radiation condition \eqref{eq:sommerfeld}. Thus the transmission problem \eqref{eq:pbtrans} rewrites
 \begin{equation}  \label{eq:pbtrans:boule} 
\left\{
\begin{aligned}
\Delta u^\pm + (k^\pm)^2 u^\pm &= 0,&  \qquad &\text{in $\Omega^\pm$,} \\
u^\I -u^\e &=  f, & \qquad  &\text{on $\Gamma$,}  \\
\partial_{\mf n}  u^\I  - \kappa  \partial_{\mf n}  u^\e &= g,& \qquad &\text{on $\Gamma$,} \\
\text{$u^\e$ satisfies \eqref{eq:sommerfeld},} && \qquad &\text{when $\vert \mf x \vert \to +\infty$.}
\end{aligned}
\right.
\end{equation}

\begin{figure}[!htbp]  

~ \hfill \includegraphics{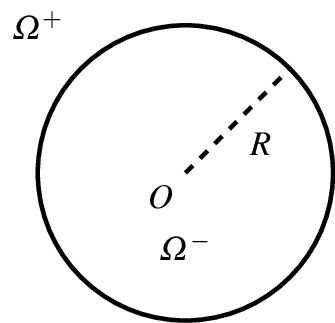} \hfill \includegraphics{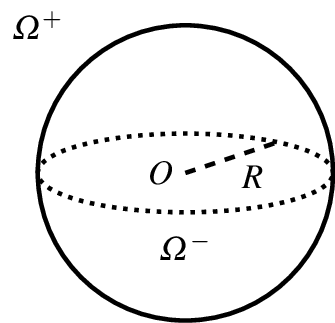} \hfill ~

\caption{\label{fig:DiskBall}Geometry of the problem \eqref{eq:pbtrans:boule} for $d=2$ (left) and $d=3$ (right)}
\end{figure}

\subsection{Reduction to linear systems}

Here we deal with Helmholtz equations in geometries with radial symmetries. Using separation of variables (we denote $(r,\theta)$ the polar coordinates in 2d and $(r,\theta,\phi)$ the spherical coordinates in 3d), it is well-known (see for instance \cite{colton2012inverse,morse1953methods}) that solutions can be expressed as series.
\begin{itemize}
\item In $\Omega^\I$ one has
\begin{equation}\label{eq:decomposition1}
\begin{aligned}
& u^\I(r,\theta) = \sum_{n \in \Z}  \frac{u^\I_n}{J_n(k^\I R)} J_n \left( k^\I r \right)  \psi_n(\theta), & \qquad & (d=2),  \\
& u^\I(r,\theta,\phi) = \sum_{\ell \in \N} \sum_{m = -\ell}^{+\ell}  \frac{u^\I_{\ell,m}}{j_\ell \left( k^\I R \right)} j_\ell \left( k^\I r \right)  \psi_{\ell,m}(\theta,\phi), & \qquad & (d=3),
\end{aligned}
\end{equation}
where $J_n$ (resp. $j_\ell$) is the Bessel function (resp. spherical Bessel function) of the first kind of order $n$ (resp. order $\ell$), $\psi_n$ the standard Fourier basis ($\psi_0 = 1/\sqrt{2\pi}$ and $\psi_n = e^{in\theta}/\sqrt{2\pi}$) and $\psi_{\ell,m}$ are the so-called spherical harmonics: 
\begin{equation}
\psi_{\ell,m}(\theta,\phi) := \sqrt{ \frac{2\ell+1}{4\pi} \frac{(\ell-m)!}{(\ell+m)!}} P_\ell^m(\cos \theta) e^{i m \phi},\qquad \ell \in \N,\ m \in \{-\ell,\ldots,\ell\} , 
\end{equation}
where $P_\ell^m$ is the associated Legendre polynomial of order $(\ell,m)$. Here, $u^\I_n$ and $u^\I_{\ell,m}$ are the modal coefficients to determine (we have normalised by $J_n(k^\I R)$ and $j_\ell(k^\I R)$ to simplify the incoming computations). 
\item In $\Omega^\I$ one has  
\begin{equation}\label{eq:decomposition2}
\begin{aligned}
& u^\e(r,\theta) = \sum_{n \in \Z} \frac{u^\e_n}{H_n   \left( k^\e R \right)}  H_n   \left( k^\e r \right)  \psi_n(\theta), & \qquad & (d=2),   \\
& u^\e(r,\theta,\phi) = \sum_{\ell \in \N} \sum_{m = -\ell}^{+\ell} \frac{u^\e_{\ell,m}}{h_\ell \left( k^\e R \right) }   h_\ell \left( k^\e r \right)  \psi_{\ell,m}(\theta,\phi), & \qquad & (d=3),
\end{aligned}
\end{equation}
where $H_n$ (resp. $h_\ell$) is the Hankel function (resp. spherical Hankel function) of first kind of order $n$ (resp. order $\ell$). Here, $u^\e_n$ and $u^\e_{\ell,m}$ are the modal coefficients to determine.
\end{itemize}
For more details about Bessel and Hankel functions, see Appendix \ref{sec:bessel} (also \cite{watson1995treatise,olver2010nist}). 

\begin{rmk}
Notice that the radiation condition of \eqref{eq:pbtrans:boule} is taken into account very simply thanks to the modal decomposition. Indeed, both $H_n$ and $h_\ell$ verifies \eqref{eq:sommerfeld} but this is not the case for the Hankel functions and spherical Hankel functions of the second kind. That is why these do not appear in \eqref{eq:decomposition2}.  
\end{rmk}

Since $(\psi_n)_n$ (resp. $(\psi_{\ell,m})_{\ell,m}$) is a Hilbert basis of $L^2(\mathbb S^1)$ (resp. $L^2(\mathbb S^2)$), plugging \eqref{eq:decomposition1} and \eqref{eq:decomposition2} in the transmission conditions of \eqref{eq:pbtrans:boule} leads to a countable family of  $2 \times 2$ linear systems. For $d=2$, we get, for all $n \in \Z$,
\begin{equation}\label{eq:systemes}
\mathcal A_n \begin{bmatrix}
 u^\I_n \\ u^\e _n
\end{bmatrix}
= 
\begin{bmatrix}
 f_n \\ g_n
\end{bmatrix}, \qquad \text{where} \qquad \mathcal A_n := 
\begin{bmatrix}
  1  &  -1     \\
\displaystyle \frac{k^\I  J^\prime_n ( k^\I R)}{ J_n ( k^\I R)}  &  - \displaystyle \kappa   \frac{k^\e H^\prime_n( k^\e R )}{ H_n  (k^\e R ) }
\end{bmatrix}.
\end{equation}
For $d=3$, we get for all $\ell \in \N$ and $m \in \{-\ell,\ldots,\ell\}$
\begin{equation}\label{eq:systemes3d}
\mathcal B_\ell 
\begin{bmatrix}
 u^\I_{\ell,m} \\ u^\e_{\ell,m} 
\end{bmatrix}
= 
\begin{bmatrix}
 f_{\ell,m}  \\ g_{\ell,m} 
\end{bmatrix}, \qquad \text{where} \qquad \mathcal B_\ell  := 
\begin{bmatrix}
   1  &  -1     \\
\displaystyle \frac{k^\I  j^\prime_n\ell ( k^\I R )}{j_\ell ( k^\I R)}   &  \displaystyle - \kappa  \frac{k^\e h^\prime_\ell( k^\e R)}{ h_\ell  (k^\e R ) }
\end{bmatrix}.
\end{equation}
The unique solvability of \eqref{eq:systemes}--\eqref{eq:systemes3d} is ensured if the determinants $\mathcal D^{(2)}_n := \det \mathcal A_n$ and $\mathcal D^{(3)}_\ell := \det  \mathcal B_\ell$ 
given by
\begin{equation}\label{eq:determinants}
\begin{aligned}
& \mathcal D^{(2)}_n = \frac{k^\I  J^\prime_n ( k^\I R )}{J_n ( k^\I R)}   - \kappa     \frac{ k^\e  H^\prime_n( k^\e R )}{ H_n  (k^\e R  ) }   ,&\qquad &n \in \Z , \\
& \mathcal D^{(3)}_\ell =  \frac{k^\I j^\prime_\ell ( k^\I R)}{ j_\ell ( k^\I R )}   - \kappa       \frac{k^\e h^\prime_\ell( k^\e R )}{h_\ell  (k^\e R )},&\qquad &\ell \in \N,
\end{aligned}
\end{equation}
never vanish.

\begin{lemme}\label{lemme:determinant}
For all $n \in \Z$ and for all $\ell \in \N$, $ \mathcal D^{(2)}_n \neq 0$ and $\mathcal D^{(3)}_\ell \neq 0$.
\end{lemme}

\begin{proof}
See Appendix \ref{sec:bessel}.
\end{proof}

Thus we can uniquely solve \eqref{eq:systemes} and \eqref{eq:systemes3d}:
\begin{equation}\label{eq:determinants:solModes2d}
\begin{bmatrix}
 u^\I_n \\ u^\e_n
\end{bmatrix} = (\mathcal A_n)^{-1}  \begin{bmatrix}
 f_n  \\ g_n
\end{bmatrix}
= \frac{1}{\mathcal D^{(2)}_n} \begin{bmatrix}
\displaystyle - \kappa  k^\e \frac{H^\prime_n( k^\e R )}{ H_n  (k^\e R ) }  & 1 \\
\displaystyle - k^\I \frac{J^\prime_n ( k^\I R)}{ J_n ( k^\I )} & 1
\end{bmatrix}
\begin{bmatrix}
 f_n  \\ g_n
\end{bmatrix}, \qquad (d=2),
\end{equation}
and
\begin{equation}\label{eq:determinants:solModes3d}
\begin{bmatrix}
 u^\I_{\ell,m} \\ u^\e_{\ell,m} 
\end{bmatrix} = (\mathcal B_\ell)^{-1}  \begin{bmatrix}
 f_{\ell,m} \\ g_{\ell,m}
\end{bmatrix}
= \frac{1}{\mathcal D^{(3)}_\ell} \begin{bmatrix}
\displaystyle - \kappa  k^\e \frac{h^\prime_n( k^\e R )}{ h_n  (k^\e R ) }  & 1 \\
\displaystyle - k^\I \frac{j^\prime_n ( k^\I R)}{ j_n ( k^\I )} & 1
\end{bmatrix}
\begin{bmatrix}
 f_{\ell,m}  \\  g_{\ell,m}
\end{bmatrix},\qquad (d=3).
\end{equation}

\subsection{Asymptotic analysis}\label{subsec:asymptoticBall}

The $u^\I_n$ and $u^\e_n$ (resp. the $u^\I_{\ell,m}$ and $u^\e_{\ell,m}$) are now uniquely determined. We want to know the regularity of the corresponding solutions $u^\I$ and $u^\e$ given by \eqref{eq:decomposition1} and \eqref{eq:decomposition2}. This regularity is linked to the rate of decaying of $u^\e_n$ and $u^\I_n$ when $n \to \pm \infty$ (resp. of $u^\e_{\ell,m}$ and $u^\I_{\ell,m}$ when $\ell \to + \infty$). Indeed, one has the following characterization of  Sobolev spaces for $s \geq 0$ (see \textit{e.g.} \cite{iorio2001fourier}):
\begin{equation}\label{eq:charasobolev1}
 \begin{aligned}
& H^s(\mathbb S^1) = \left\{ u \in L^2(\mathbb S^1)\ \colon  \sum_{n \in \Z} \left(1+n^{2}\right)^s \left \vert u_n \right\vert^2 <  +\infty\right \}, \\
& H^s (\mathbb S^2) = \left\{ u \in L^2(\mathbb S^2)\ \colon \sum_{\ell \in \N} \sum_{m = -\ell}^{+\ell} \left(1+\ell^{2}\right)^s \left\vert u_{\ell,m} \right\vert^2 <  +\infty \right\},
 \end{aligned} 
\end{equation}
where $u_n := \langle u,\psi_n \rangle_{L^2(\mathbb S^1)}$ and $u_{\ell,m} := \langle\psi,\psi_{\ell,m} \rangle_{L^2(\mathbb S^2)}$. These definitions can be extended by duality to negative exponents:
\begin{equation}\label{eq:charasobolev2}
\begin{aligned}
& H^{-s}(\mathbb S^1) = \left\{ \phi\in \mathcal C^\infty(\mathbb S^1)^* \colon  \sum_{n \in \Z} \left(1+n^{2}\right)^{-s}  \vert \phi_n \vert^2 <  +\infty \right\}, \\
& H^{-s} (\mathbb S^2) = \left\{ \phi \in \mathcal C^\infty(\mathbb S^2)^* \colon \sum_{\ell \in \N} \sum_{m = -\ell}^{+\ell} \left(1+\ell^{2}\right)^{-s}\vert \phi_{\ell,m} \vert^2 <  +\infty \right\},
 \end{aligned} 
\end{equation}
where $\phi_n := \phi(\overline{\psi_{n}}) = \phi(\psi_{-n})$ and $\phi_{\ell,m}  := \phi(\overline{\psi_{\ell,m}})$. 

In the classical case of a transmission between two positive materials, it is enough to perform an asymptotic at order 0 to be able to conclude. For our problem, it is necessary to go further because the first terms of the asymptotic may cancel. Before doing the asymptotic analysis, the first thing to notice is that, for the 2d case, $J_{-n}(\cdot) = (-1)^n J_n(\cdot)$ and $H_{-n}(\cdot) = (-1)^n H_n(\cdot)$ for all $n \in \Z$ thus one just need to treat the case $n \to +\infty$. Moreover, we need some asymptotics of Bessel and Hankel functions:
\begin{prop}
Let $r>0$ and $N \in \N^*$. One has the asymptotics when $n \to + \infty$:
\begin{equation}\label{eq:asymptobessel}
\begin{aligned}
& J_n(r)  = \frac{r^n}{2^n n! } \left[ \sum_{k=0}^N \frac{ (-1)^k n! }{k!(n+k)!} \left( \frac{r}{2} \right)^{2k}+ \mathcal O\left(\frac{1}{n^{N+1}} \right)\right], \\
& J^\prime_n(r) = \frac{r^{n-1}}{2^n (n-1)! } \left[ \sum_{k=0}^N \frac{ (-1)^k (n+2k) (n-1)!}{k!(n+k)!} \left( \frac{r}{2} \right)^{2k}+ \mathcal O\left(\frac{1}{n^{N+1}} \right)\right], \\
& H_n(r)  = \frac{-i}{ \pi} \frac{2^n (n-1)!}{r^n} \left[\sum_{k = 0}^N \frac{ (n-k-1)!}{k!(n-1)!}  \left( \frac{r}{2} \right)^{2k}+\mathcal O\left(\frac{1}{n^{N+1}} \right)\right], \\
& H^\prime_n(r)  =  \frac{i}{\pi} \frac{2^n n!} {r^{n+1}} \left[\sum_{k = 0}^N \frac{ (n-2k)(n-k-1)!}{k!n!}    \left( \frac{r}{2} \right)^{2k}+\mathcal O\left(\frac{1}{n^{N+1}} \right)\right],
\end{aligned}
\end{equation}
and when $\ell \to +\infty$:
\begin{equation}\label{eq:asymptobessel3d}
\begin{aligned}
& j_\ell (r)  = \frac{r^\ell}{(2\ell+1)!!} \left[ \sum_{k=0}^N \frac{ (-1)^k (2\ell+1)!!}{k!(2\ell+2k+1)!!}  \left( \frac{r^2}{2} \right)^k+ \mathcal O\left(\frac{1}{\ell^{N+1}} \right)\right], \\
& j^\prime_\ell (r)  = \frac{\ell r^{\ell-1}}{(2\ell+1)!!} \left[ \sum_{k=0}^N \frac{ (-1)^k(\ell+2k)(2\ell+1)!! }{k!(2\ell+2k+1)!! \ell} \left( \frac{r^2}{2} \right)^k+ \mathcal O\left(\frac{1}{\ell^{N+1}} \right)\right], \\
& h_\ell (r)  = -i \frac{(2\ell-1)!!}{r^{\ell+1}} \left[ \sum_{k=0}^N \frac{ (2\ell-2k-1)!!}{k!(2\ell-1)!! }  \left( \frac{r^2}{2} \right)^k+ \mathcal O\left(\frac{1}{\ell^{N+1}} \right)\right], \\
& h^\prime_\ell (r)  = i  \frac{ (\ell+1) (2\ell-1)!!}{ r^{\ell+2}} \left[ \sum_{k=0}^N \frac{ (\ell+1-2k)(2\ell-2k-1)!!}{k!(\ell+1) (2\ell-1)!! }\left( \frac{r^2}{2} \right)^k+ \mathcal O\left(\frac{1}{\ell^{N+1}} \right)\right],
\end{aligned}
\end{equation}
where $!!$ stands for the double factorial, defined by $0!! = 1$, $p!! = 2 \times 4 \times \cdots \times p $ for $p \in \{2,4,6,\ldots\}$ and $p!! = 1 \times 3 \times \cdots \times p$ for $p \in \{1,3,5,\ldots\}$.
\end{prop}

\begin{proof}
See Appendix \ref{sec:bessel}.
\end{proof}
 
We can now give the asymptotics of the determinants $\mathcal D^{(2)}_n$ and $\mathcal D^{(3)}_\ell$:
\begin{prop}
One has
\begin{equation}\label{eq:asympDet2d}
\mathcal D^{(2)}_n \underset{n \to +\infty}{\sim}   \left\{
\begin{aligned}
\frac{1+\kappa}{R}\, n   & \qquad \text{if $\kappa  \neq -1$,}\\
R \left( (k^\e)^2-(k^\I)^2 \right)\, n^{-1}   & \qquad \text{if $\kappa  =-1$ and $k^\e \neq k^\I$,} \\
R  (k^\e)^2\,   n^{-2}  & \qquad \text{if $\kappa  =-1$ and $k^\e = k^\I$,} 
\end{aligned} \right.
\end{equation}
and 
\begin{equation}\label{eq:asympDet3d}
\mathcal D^{(3)}_\ell \underset{\ell \to +\infty}{\sim}   \left\{
\begin{aligned}
\frac{1+\kappa}{R}\ \ell    & \qquad \text{if $\kappa  \neq -1$,}\\
\frac{-1}{R}    & \qquad \text{if $\kappa  =-1$.}
\end{aligned} \right.
\end{equation}
\end{prop}

\begin{proof}
Plugging \eqref{eq:asymptobessel} and \eqref{eq:asymptobessel3d} in the definitions \eqref{eq:determinants} of $\mathcal D^{(2)}_n$ and $\mathcal D^{(3)}_\ell$, one gets the results after tedious but straightforward calculations. 
\end{proof}

Thanks to \eqref{eq:determinants:solModes2d} (resp. \eqref{eq:determinants:solModes3d}), we can now deduce the asymptotics of $u_n^\I$ and $u_n^\e$ (resp. of $u_{\ell,m}^\I$ and $u_{\ell,m}^\e$):
\begin{prop}\label{prop:finalasymptotball}
For $d=2$, one has
\begin{equation}
\begin{bmatrix}
 u^\I_n \\ u^\e_n
\end{bmatrix} 
\underset{n \to +\infty}{\sim} \left\{ \begin{aligned}
\frac{1}{\kappa+1}\,\mathcal M_{n,\kappa}(0)
\begin{bmatrix}
 f_n  \\ g_n
\end{bmatrix} & \qquad \text{if $\kappa  \neq -1$,}\\ 
\frac{2}{R[(k^\e)^2-(k^\I)^2]}\,\mathcal M_{n,-1}(2)
\begin{bmatrix}
 f_n  \\ g_n
\end{bmatrix} & \qquad \text{if $\kappa  =-1$ and $k^\e \neq k^\I$,} \\
\frac{1}{R^2 (k^\e)^2}\,\mathcal M_{n,-1}(3)
\begin{bmatrix}
 f_n  \\ g_n
\end{bmatrix} & \qquad \text{if $\kappa  =-1$ and $k^\e = k^\I$,} \\
\end{aligned} \right. 
\end{equation}
and for $d=3$
\begin{equation}
\begin{bmatrix}
 u^\I_{\ell,m} \\  u^\e_{\ell,m} 
\end{bmatrix} 
\underset{\substack{\ell \to +\infty \\ -\ell \leq m \leq \ell}}{\sim} \left\{ \begin{aligned}
\frac{1}{\kappa+1}\,\mathcal M_{\ell,\kappa}(0)
\begin{bmatrix}
 f_{\ell,m}  \\ g_{\ell,m}
\end{bmatrix} & \qquad \text{if $\kappa  \neq -1$,}\\ 
   - \mathcal M_{\ell,-1}(1)
\begin{bmatrix}
 f_{\ell,m}  \\ g_{\ell,m}
\end{bmatrix}  & \qquad \text{if $\kappa  =-1$,}
\end{aligned} \right.
\end{equation}
where $\mathcal M_{m,\kappa}(p)$ is the matrix
\begin{equation}
\mathcal M_{m,\kappa}(p) := \begin{bmatrix}
\kappa \, m^p & R\,m^{p-1} \\ - m^p & R\,m^{p-1}
\end{bmatrix}.
\end{equation}
\end{prop}

\subsection{Conclusion}
\label{sec:conclusion}

Thanks to the introduction of the matrix $\mathcal M_{m,\kappa}(k)$, it is really easy to read the asymptotics of $(u^\I_n , u^\e_n)$ and $(u^\I_{\ell,m},u^\e_{\ell,m})$ in term of the ones for $(f_n , g_n)$ and $(f_{\ell,m},g_{\ell,m})$. For instance, in dimension $d=2$, both $u^\I_n$ and $u^\e_n$ are equivalent to $C_1n^p f_n + C_2n^{p-1} g_n$ where $C_1$ and $C_2$ are non-zero constants. Thus, using the characterisations of Sobolev spaces \eqref{eq:charasobolev1} and \eqref{eq:charasobolev2}, we can give the final result of this section:
\begin{thm}\label{thm:ball}
Let $s>0$ be fixed. For $(f,g) \in\mf H^{s+p}(\Gamma)$, \eqref{eq:pbtrans:boule} has a unique solution $(u^\I\vert_\Gamma,u^\e\vert_\Gamma) \in H^s(\Gamma)^2$ where $p \in \N$ is called the \emph{order of regularity lost} and is given by the Table \ref{table:Ball}.
\end{thm}

This result is optimal in the sense that if $(f,g) \in H^{s+p}(\Gamma)$ but not in $(f,g) \in H^{s+p+\eps}$ for all $\eps>0$ then one cannot expect a better regularity than $(u^\I\vert_\Gamma,u^\e\vert_\Gamma) \in H^s(\Gamma)^2$.We recover the results of \cite{ola1995remarks} and \cite{nguyen2016limiting} for the dimension $d \geq 3$ when $\Omega^+$ is strictly convex.

\begin{rmk}\label{rmk:dim}
Actually one can do the same computations in any dimension $d \geq 3$ using generalised spherical harmonics (\cite{stein1971introduction}) and generalised spherical Bessel and Hankel functions $r^{1-d/2} J_{n+1-d/2}(r)$ and $r^{1-d/2} H_{n+1-d/2}(r)$. One can show that the conclusion of Theorem \ref{thm:ball} for $d> 3$ are the same as the ones for $d = 3$ (the only particular case is $d=2$).
\end{rmk}

One can reinterpret the conclusion of Theorem \ref{thm:ball} in term of external source, that is to say the original Helmholtz equation \eqref{eq:TE} becomes $ \nabla \cdot (\sigma \nabla u)  + \omega^2 \varsigma u =F$ where $F \in L^2(\Omega^\e)$. By standard regularity results, $u^\mathrm{inc}$ (defined now as the solution of $ \Delta u^\mathrm{inc} + (k^\e)^2 u^\mathrm{inc}= F$) has a $H^2$ regularity, so $(f,g) \in \mf H^{3/2}(\Gamma)$. Using Theorem \ref{thm:ball}, $u$ has the standard $H^2$ regularity (outside $\Gamma$) for the classical case but is less regular in the other cases. In dimension $d=2$, $u$ has only a $L^2$ regularity for the critical case and a $H^{-1}$ regularity for the super-critical case. For $d \geq 3$, $u$ has only a $H^1$ regularity for both the critical and the super-critical case.

\begin{rmk}\label{rmk:source}
One could argue that theses loses of regularity does not matter in practice, since $(f,g)$ often belongs to $\mathcal C^\infty(\Gamma)^2$ because, from \eqref{eq:pbondeinc}, $u^\mathrm{inc} \vert_\Gamma$ is smooth by standard regularity results (in the case of an external source $ \Delta u^\mathrm{inc} + (k^\e)^2 u^\mathrm{inc}= F$, this is true as soon as the support of $F$ is compactly embedded in $\Omega^\e$). As a consequence, $(u^\I\vert_\Gamma,u^\e\vert_\Gamma)$ belongs to $\mathcal C^\infty(\Gamma)^2$ too. However, the loses of regularity coming from the change of sign have an impact on numerical methods: the standard $H^1$ functional framework does not applies here when $\kappa = -1$ thus convergence of standard numerical method (for instance finite elements) are not ensured. We refer to \cite{carvalho2015etude} and references therein for more details on these issues. See also the end of Section \ref{subsubsec:conclusion} for a case where $(f,g)\in \mathcal C^\infty(\Gamma)^2$ is not enough to ensure smoothness of $(u^\I\vert_\Gamma,u^\e\vert_\Gamma)$.
\end{rmk}

\begin{table} 
\begin{center}
\begin{tabular}{|c|c|c|}
\hline
 &   $d=2$  & $d \geq 3$ \\
 \hline
 standard case $\kappa \neq -1$  & $ p = 0$ & $ p = 0$ \\
 \hline
 critical case $\kappa  = -1$ and $k^\e \neq k^\I$ & $ p = 2$ & $ p = 1$ \\
\hline 
 super-critical case $\kappa  =  -1$ and $k^\e = k^\I$  & $p = 3$  & $ p = 1$\\
 \hline
\end{tabular}
\end{center}
\caption{\label{table:Ball} Values of $p$ that appear in Theorem \ref{thm:ball} (see Remark \ref{rmk:dim} for $d>3$).}
\end{table}

\subsection{Numerical validations}

In order to verify the asymptotics given in Proposition \ref{prop:finalasymptotball}, we compute numerically the inverses of the matrices $\mathcal A_n$ and $\mathcal B_\ell$ defined in \eqref{eq:systemes} and \eqref{eq:systemes3d} using the MATLAB software for $n = \ell = 1,\ldots,100$, $R = 1$. For the standard case, we use $\kappa = -3$ and $k^\e = k^\I = 2$; for the critical case, we use $\kappa = -1$, $k^\e=1$ and $k^\I = 3$ and for the super-critical case, we use $\kappa = -1$ and $k^\e = k^\I = 2$.

The results are shown in Figure \ref{fig:asymptotics} in log-log scale. More precisely we plot (the logarithm of) the values of the entries of $(\mathcal A_n)^{-1}$ and $(\mathcal B_\ell)^{-1}$ as functions of (the logarithm of) $n$ and $\ell$ respectively. We recover the claimed results of Proposition \ref{prop:finalasymptotball}: the slopes of the curves are the same as the values of $p$ in $\mathcal M_{n,\kappa}(p)$ and $\mathcal M_{\ell,\kappa}(p)$ in each different case.

\begin{figure}[!htbp]  
\centering
\includegraphics[width=0.49\textwidth]{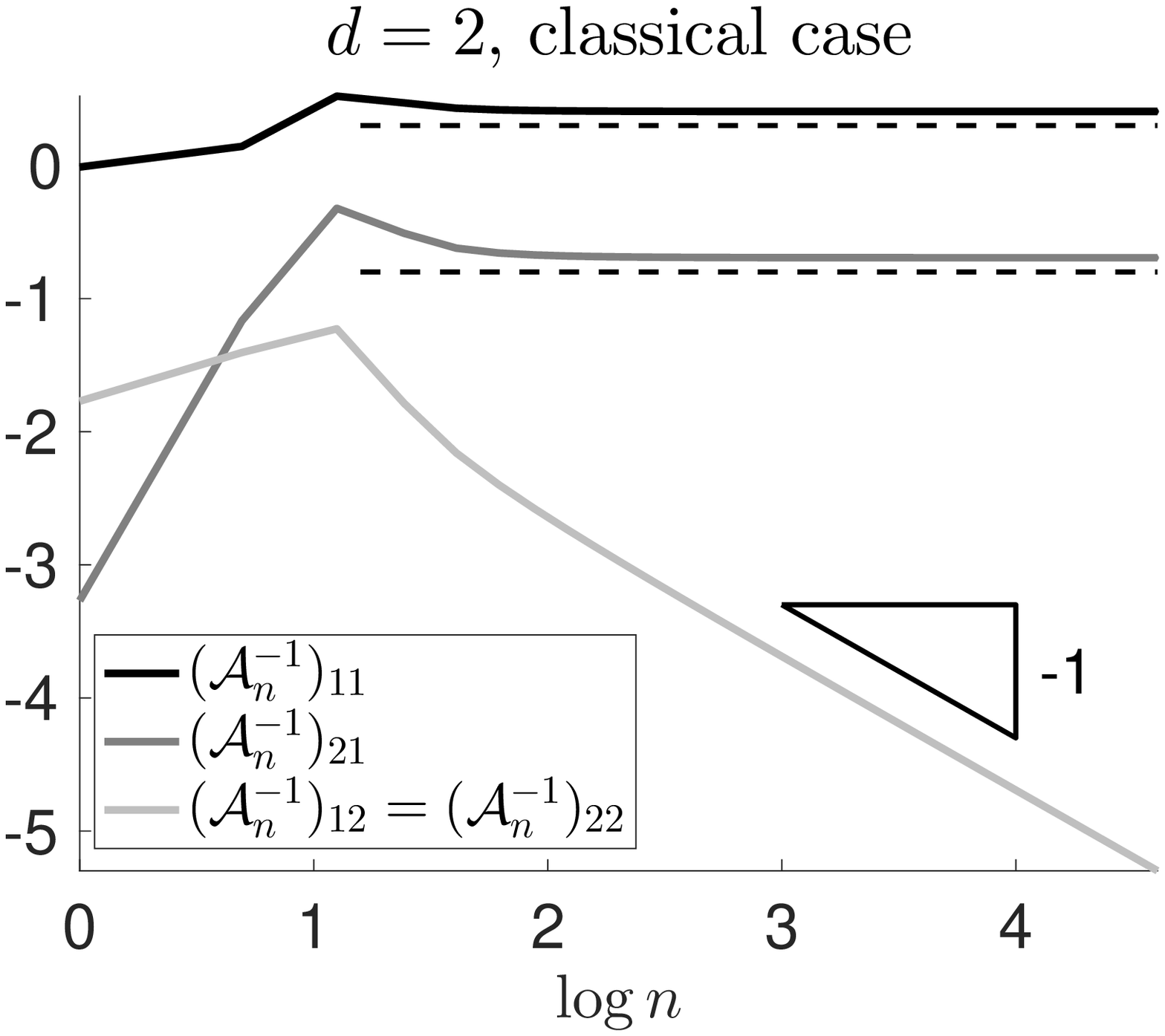} \hfill \includegraphics[width=0.49\textwidth]{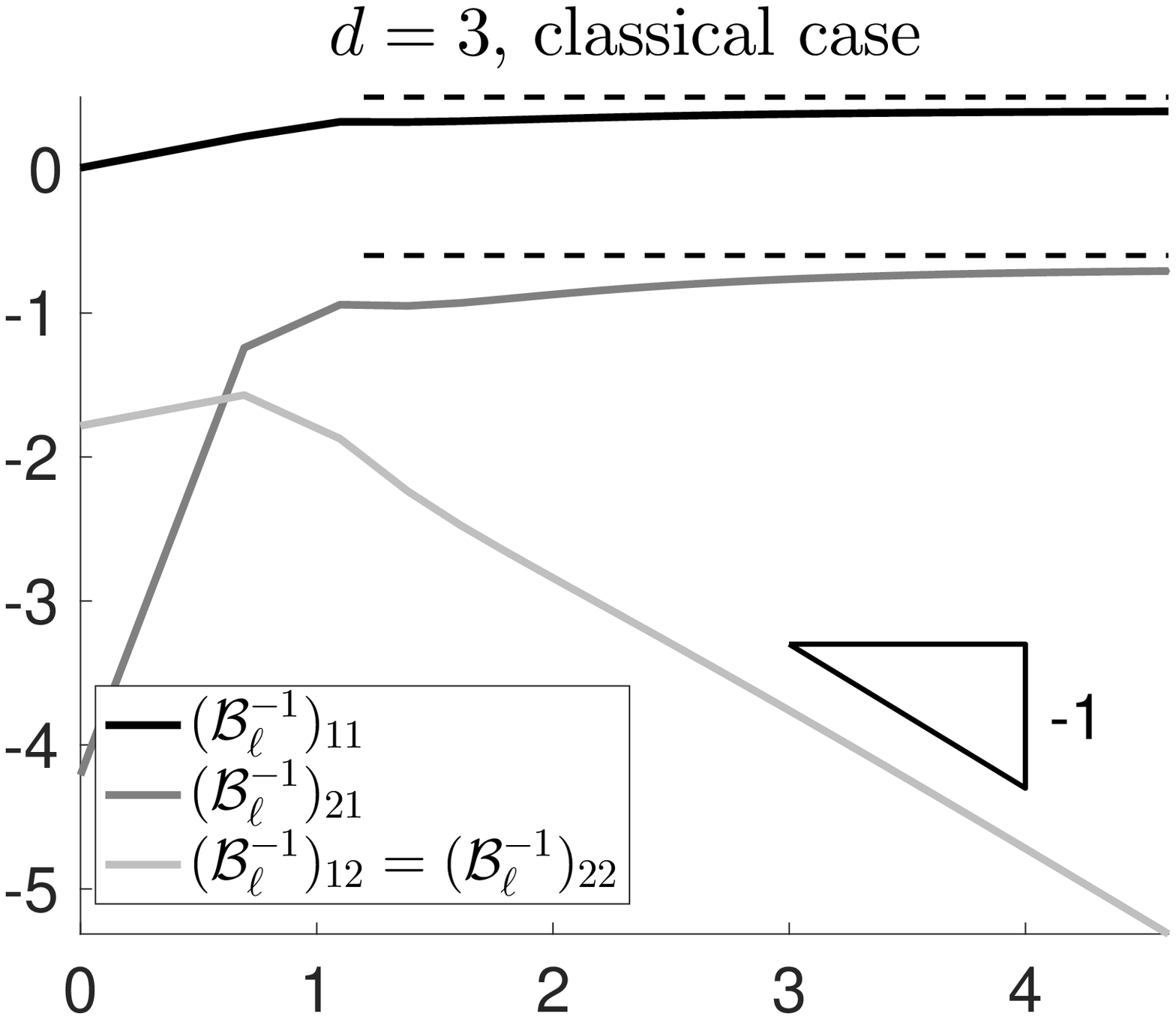} \\
\includegraphics[width=0.49\textwidth]{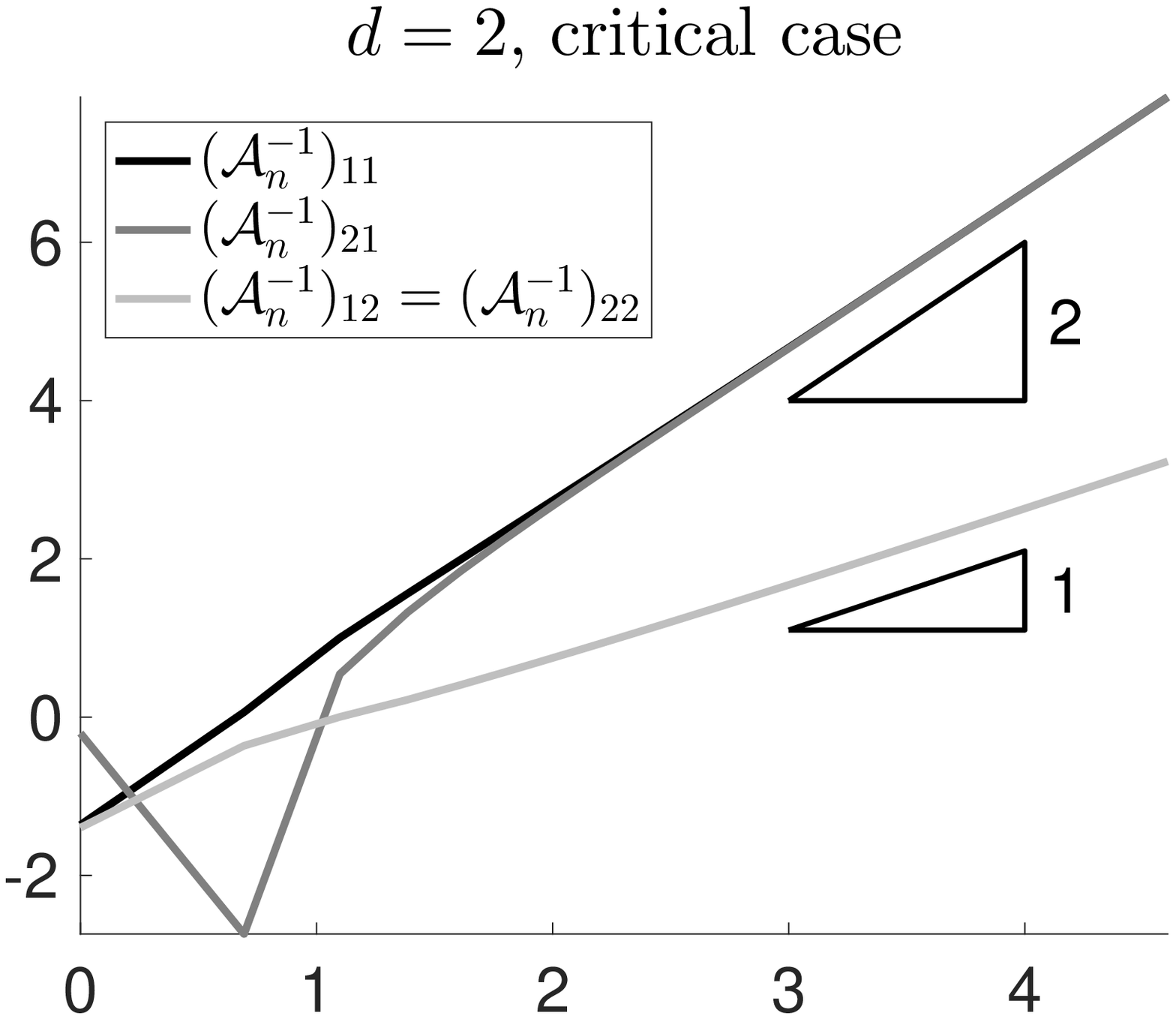} \hfill \includegraphics[width=0.49\textwidth]{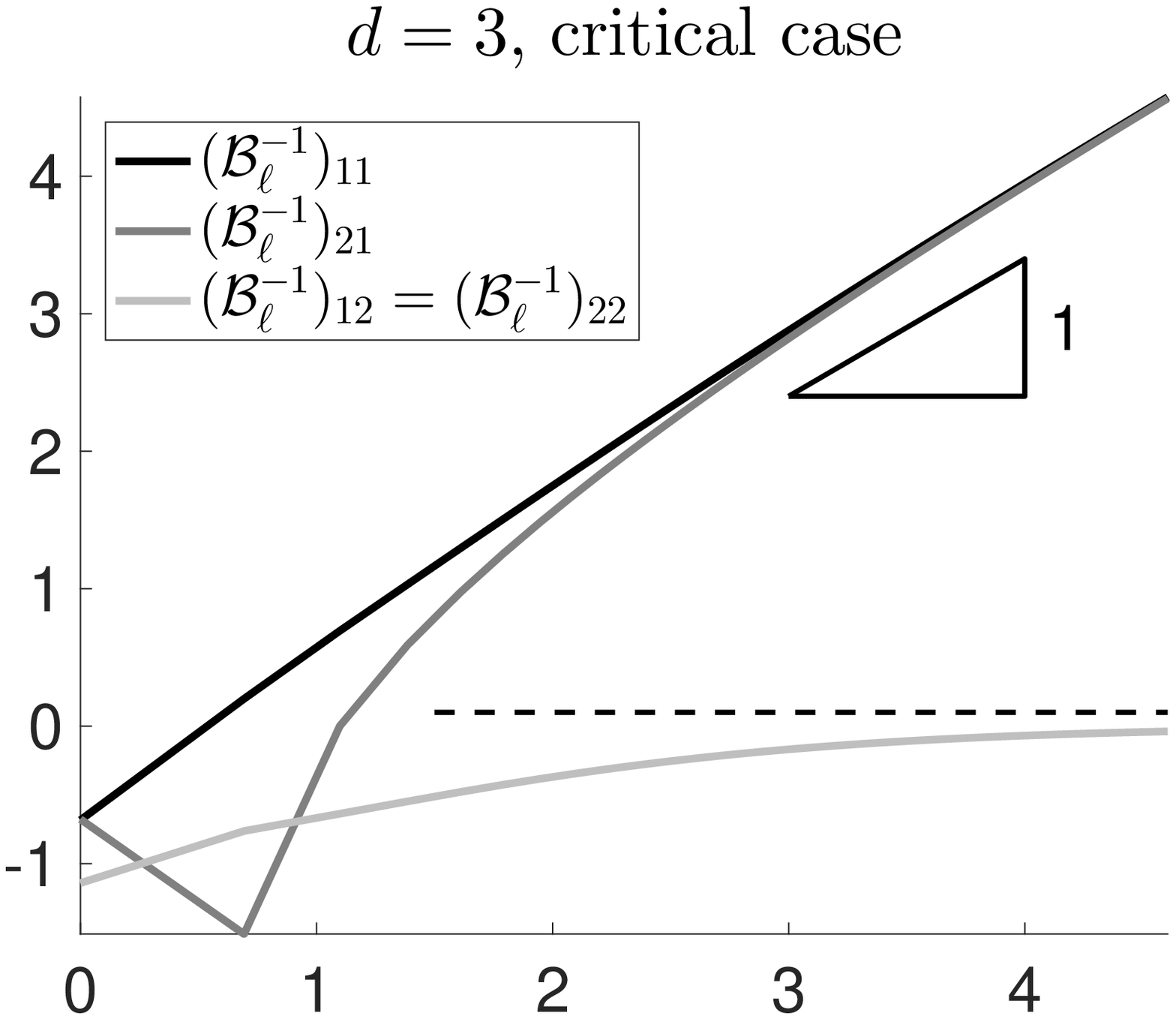} \\
\includegraphics[width=0.49\textwidth]{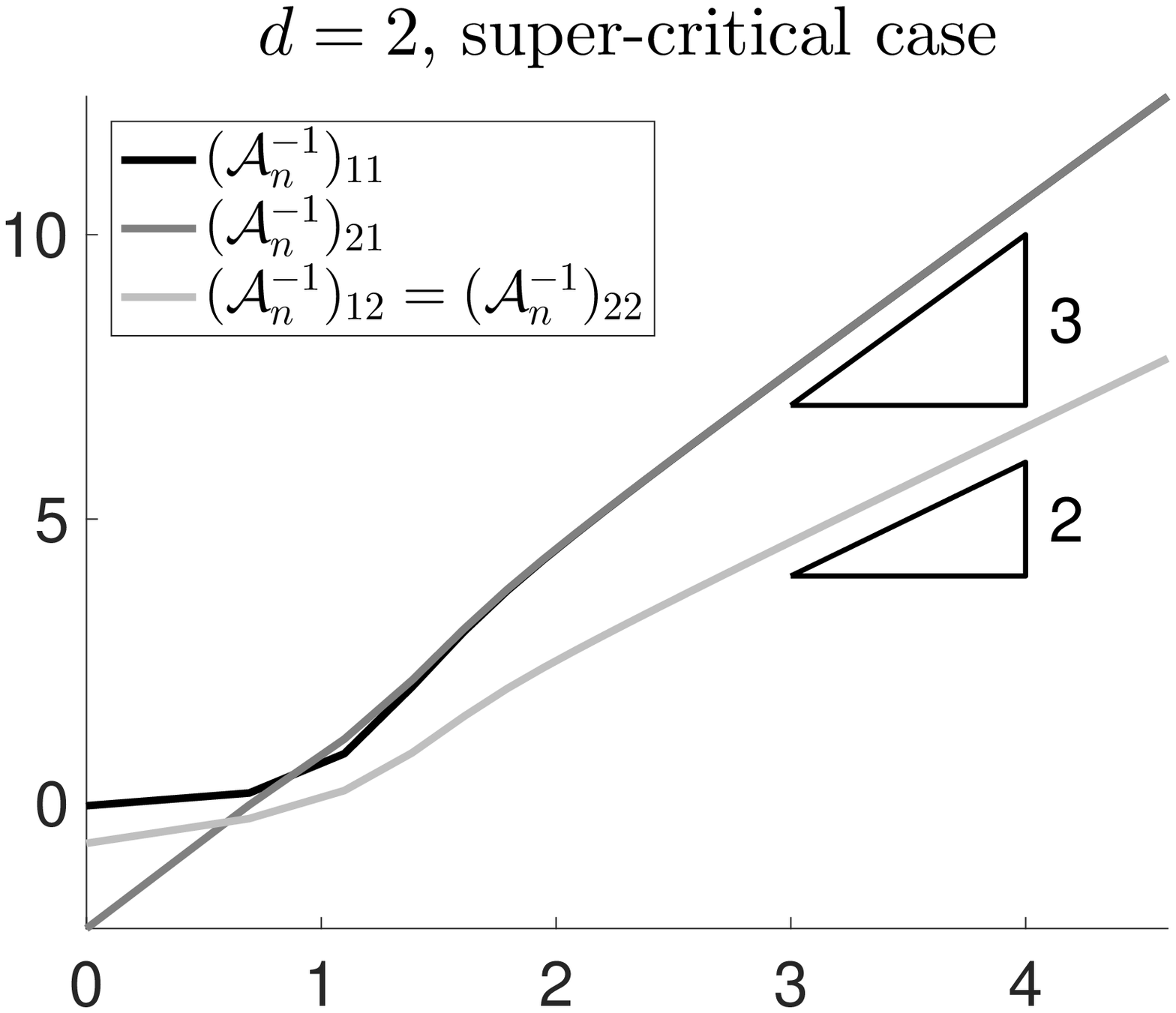} \hfill \includegraphics[width=0.49\textwidth]{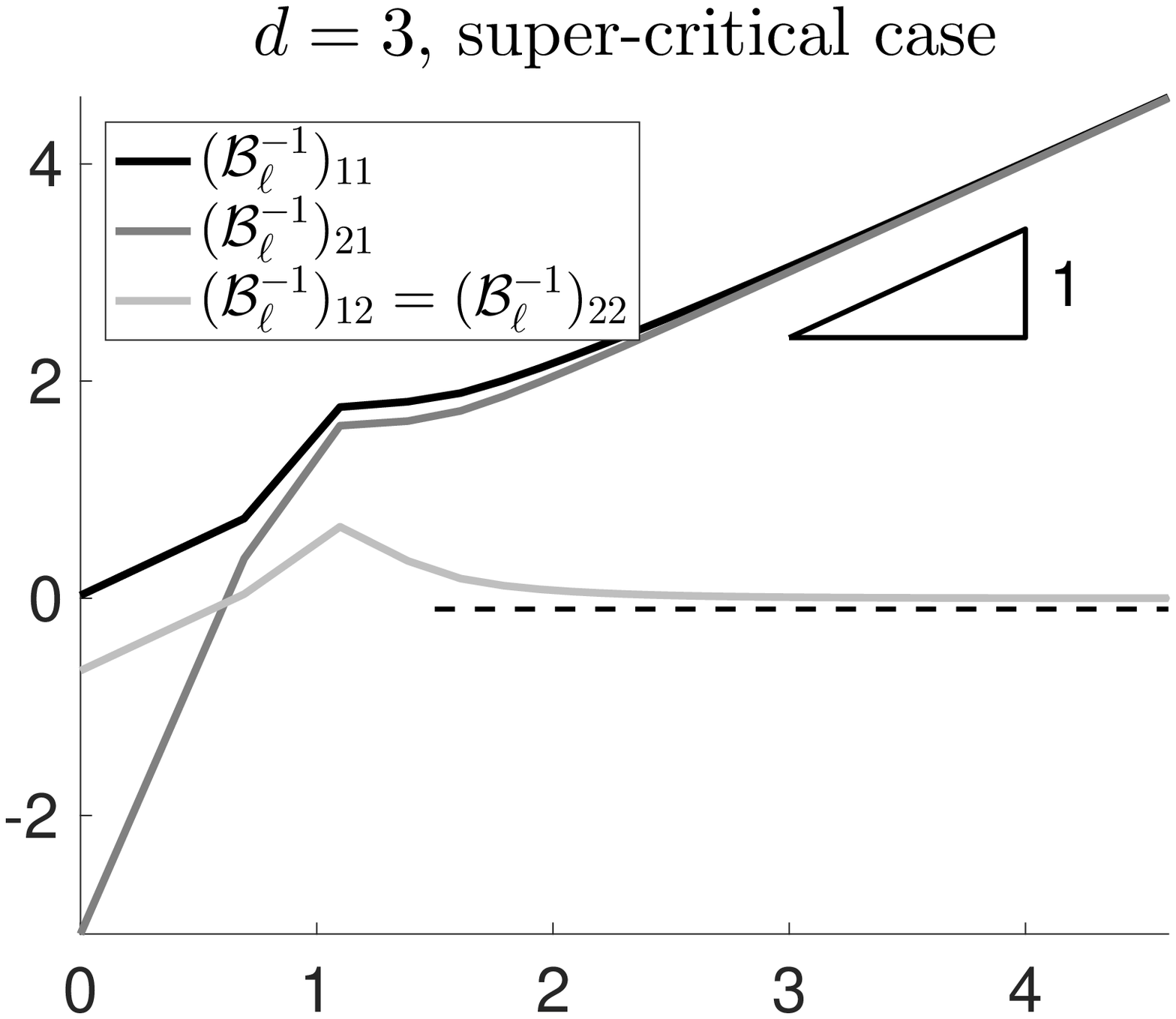} 
\caption{\label{fig:asymptotics}Plot of the entries of $(\mathcal A_n)^{-1}$ and $(\mathcal B_\ell)^{-1}$ as functions of $n$ and $\ell$ respectively, in log-log scale (notice that $(\mathcal A_n^{-1})_{12} = (\mathcal A_n^{-1})_{22}$ and $(\mathcal B_\ell^{-1})_{12} = (\mathcal B_\ell^{-1})_{22}$).}
\end{figure}

\subsection{When the curvature degenerates}

One could ask what happens when the radius $R$ tends to $+\infty$, namely when the curvature tends to 0, since it was pointed out in \cite{ola1995remarks} and \cite{nguyen2016limiting}  that the strict convexity of $\Omega^+$ plays a central role. We focus on what happens for the dimension $d=2$, similar results holds for $d \geq 3$. If one takes directly the limit $R \to +\infty$ in \eqref{eq:determinants:solModes2d} with fixed $n$, nothing interesting occurs. This is due to the fact that one needs to scale $n$ according to $R$, otherwise the limit problem could be seen as a zero-frequency problem. More precisely we must impose that the ratio between $n$ and $R$ remains constant. Doing so, one gets the following result:
\begin{prop}\label{prop:limitR}
Let $n \in \N^*$ and $R>0$ be such that the ratio $\xi := n/R$ is fixed and verifies $\xi > \max(k^\I,k^\e)$. Then one has
\begin{equation}\label{eq:limitR}
\mathcal D^{(2)}_n(R)  \underset{ \substack{R \to +\infty \\ n = R \xi}}{\longrightarrow} \sqrt{\xi^2-(k^\I)^2}+\kappa \sqrt{\xi^2-(k^\e)^2}.
\end{equation}
\end{prop}

In particular, the limit value in \eqref{eq:limitR} is not zero for the standard case $\kappa\neq -1$ and the critical case ($\kappa = -1$ and $k^\I \neq k^\e$), except maybe for one value of $\xi$, and vanishes for all $\xi > \max(k^\I,k^\e)$ in the super-critical case ($\kappa = -1$ and $k^\I = k^\e$). In this last case, at the limit $R \to +\infty$, the systems \eqref{eq:determinants:solModes2d} become non-invertible.

\begin{proof}
Since $\xi > \max(k^\I,k^\e)$, one could define  $\alpha := \sech^{-1}(k^\I/\xi)$ and $\beta := \sech^{-1}(k^\e/\xi)$. Thus one has
\begin{equation}
\mathcal D^{(2)}_n\left( \frac{n}{\xi} \right) = k^\I \frac{J^\prime_n(n \sech \alpha)}{J_n(n \sech \alpha)} + \kappa k^\e \frac{H^\prime_n(n \sech \beta)}{H_n(n \sech \beta)}.
\end{equation}
Using Debye's expansions (\cite{watson1995treatise,olver2010nist}), one has
\begin{equation}
\begin{aligned}
& J_n(n \sech \alpha) \underset{n \to +\infty}{\sim} \frac{e^{n(\tanh \alpha - \alpha)}}{\sqrt{2 \pi n \tanh \alpha}}, \quad 
J^\prime_n(n \sech \alpha) \underset{n \to +\infty}{\sim} e^{n(\tanh \alpha - \alpha)} \sqrt{\frac{\sinh 2 \alpha}{4 \pi n}}, \\
& H_n(n \sech \beta) \underset{n \to +\infty}{\sim} \frac{-i e^{n( \beta-\tanh \beta)}}{\sqrt{ \frac\pi2 n \tanh \beta}} \quad \text{and} \quad 
H^\prime_n(n \sech \beta) \underset{n \to +\infty}{\sim} e^{i n( \beta-\tanh \beta)} \sqrt{\frac{\sinh 2 \beta}{ \pi n}}.
\end{aligned}
\end{equation}
By standard hyperbolic trigonometric identities, one gets
\begin{equation}\label{eq:limitRtmp1}
\mathcal D^{(2)}_n\left( \frac{n}{\xi} \right) =
\frac{k^\I}{\sqrt{2}} \sqrt{\sinh 2\alpha\tanh \alpha} + \kappa \frac{k^\e}{\sqrt{2}} \sqrt{\sinh 2\beta \tanh \beta} 
 = k^\I \sinh \alpha - \kappa k^\e \sinh \beta.
\end{equation}
Now, using $\sinh \sech^{-1} z = \frac{\sqrt{1-z^2}}{z}$, $z \in (0,1)$, gives us
\begin{equation}
\sinh \alpha = \sinh \sech^{-1} \left( \frac{k^\I}{\xi} \right) = \frac{\sqrt{\xi^2-(k^\I)^2}}{k^\I} \quad \text{and} \quad
\sinh \beta = \frac{\sqrt{\xi^2-(k^\e)^2}}{k^\e}.
\end{equation}
Plugging this in \eqref{eq:limitRtmp1} gives \eqref{eq:limitR}. 
\end{proof}

\begin{rmk}\label{rmk:limitR}
The variable $\xi$ in Proposition \ref{prop:limitR} plays the role of the Fourier variable of a limit problem that is a transmission problem between two half-planes. The conditions $\xi > \max(k^\I,k^\e)$ means that we deal with evanescent waves. These facts must be linked to some results of Section \ref{sec:planar} (see Remark \ref{rmk:limitR2}).
\end{rmk}

\section{Some cases with flat interfaces}
\label{sec:planar}

Proposition \ref{prop:limitR} shows additional difficulties may appear when the curvature of the interface $\Gamma$ tends to 0, \textit{i.e.} when $\Gamma$ becomes flat. We shall now investigate more on this case. In order to stay in the pleasant framework of modal decomposition, we deal with waveguides. More precisely, now the dimension is $d \geq 2$ ($d$ can be greater than 3). We define a waveguide $\mathcal B := \R \times \Gamma $ where $\Gamma$ is a non-empty bounded connected open set of $\R^{d-1}$ with Lipschitz boundary. 

In the following, $x \in \R$ denotes the variable in the longitudinal direction and $\mf y \in \R^{d-1}$ the variables in the transverse section.

\begin{rmk}
Here we chose not to consider the case where $\Omega^\e$ and $\Omega^\I$ are half-spaces in order to avoid technical difficulties (that appear even without changes of sign): the standard technique would be to perform a Fourier transform with respect to $\mf y$. But since we are dealing with unbounded domains, solutions are not in $L^2$ and the radiation conditions to impose are not straightforward any more. It would require to use involved tools like generalised Fourier transforms (beyond the scope of this paper, see for instance \cite{weder2012spectral} that deals with perturbed stratified media or \cite{dhia2009diffraction} for perturbed open waveguides). 
\end{rmk}

Using separation of variables, one can show that a solution $u$ of the Helmholtz equation $\Delta u + k^2 u=0$ on $\mathcal B$ with some Boundary Conditions (BCs) on $\partial B$ that does not depend on $x$ (thus it is sufficient to impose them on $\partial \Gamma$) can be expressed as
\begin{equation}\label{eq:series}
u(x,\mf y) = \sum_{n \in \N} u_n e^{u \beta_n^+ x} \psi_n(\mf y),\qquad x \in \R,\ \mf y \in \Gamma.
\end{equation}
Here, the $(\psi_n)_n$ are the eigenfunctions of the 
standard eigenvalue problem:
 \begin{equation} \label{eq:vplaplacien} 
\left\{
\begin{aligned}
- \Delta_{\mf y} \psi & =  \lambda \psi, &  \quad &  \text{in $\Gamma$,}\\
\text{+ BCs} & & \quad &  \text{on $\partial \Gamma$.}\\
\end{aligned}
\right.
\end{equation}
We shall stay rather vague about the boundary conditions, but in order to perform a modal analysis, we have to suppose that they are choosen such that the operator $\Delta_{\mf y}$  is self-adjoint with compact resolvent (\cite{davies1996spectral}). For instance, this is the case for homogeneous Dirichlet or Neumann conditions. We assume that it is the case in the following. Then, the problem \eqref{eq:vplaplacien} admits a countable number of non-trivial solutions $(\lambda_n,\psi_n)$ where the $\lambda_n$ are the positive eigenvalues of finite multiplicity tending to $+\infty$ and the associated eigenfunctions $(\psi_n)_n$ form a Hilbert basis of $L^2(\Gamma)$. 

The $\beta_n^\pm$ in \eqref{eq:series} are solution of $(\beta_n)^2 = k^2 - \lambda_n$. We make the following choices for the square roots: we set, for all $n \in \N$,
\begin{equation}\label{eq:choixbeta}
\begin{aligned}
& \beta^\e_n := \begin{cases} \sqrt{ (k^\e)^2 - \lambda_n } &  \text{if $\lambda_n < (k^\e)^2 $ } \\
i  \sqrt{  \lambda_n - (k^\e)^2 } &  \text{if $\lambda_n > (k^\e)^2$ } \end{cases} \\
&  \beta^\I_n := \begin{cases} \sqrt{ (k^\I)^2 - \lambda_n } & \text{if $\lambda_n < (k^\I)^2$ } \\
-i  \sqrt{  \lambda_n - (k^\I)^2 } &  \text{if $\lambda_n > (k^\I)^2$.}  \end{cases}
\end{aligned}
\end{equation}

\begin{rmk}\label{rmk:cutoff0}
In order to avoid some technical issues that are intrinsic to waveguides but have nothing to do with the changes of sign, we suppose that $k^\e$ and $k^\I$ are not cut-off wave numbers, that is to say $\beta^\e_n  \neq 0$ et $\beta^\I_n \neq 0$ for all $n \in \N$, or equivalently $(k^\I)^2 \neq \lambda_n$ and $(k^\e)^2 \neq \lambda_n$ for all $n \in \N$. This could happen only for a finite numbers of $\beta^\e_n $ and $\beta^\I_n$ and does not change the conclusion of Theorems \ref{thm:guide2} and \ref{thm:guide} (see also Remark \ref{rmk:cutoff}).
\end{rmk}

\subsection{A case where the negative material is unbounded}
\label{sec:The unbounded case}

We first consider the case $\Omega^\I = (0,+\infty) \times \Gamma$. Its exterior is then $\Omega^\e = (-\infty,0) \times \Gamma$ (see Figure \ref{fig:pbguide:nd2}). 
\begin{figure}[!htbp]   
\centering
\includegraphics[width=0.75\textwidth]{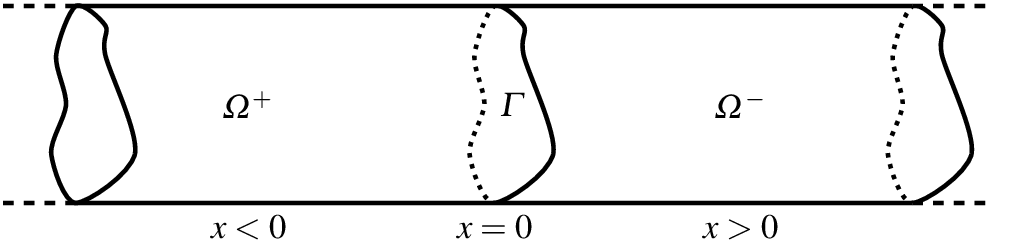}
\caption{\label{fig:pbguide:nd2}Geometry of the problem \eqref{eq:pbguide:nd2}.}
\end{figure}
Since $\Omega^\I$ is not bounded, one need to impose a radiation condition when $x$ tends to $+\infty$ but since $\Omega^\I$ is a negative material, the ``correct'' (\textit{i.e.} physically relevant) radiation condition is not the usual one. One can show that, due to the presence of negative coefficients, the radiation condition \eqref{eq:sommerfeld} is now (notice the change of sign)
\begin{equation}\label{eq:sommerfeldInverse}
\lim_{R \to + \infty} \int_{\vert \mf x \vert = R} \left\vert \frac{\partial }{\partial r}u + i k^\I u \right \vert^2 \diff r =0.
\end{equation}
For a justification of this, see the Appendix \ref{sec:radiationcondition} (see also \cite{malyuzhinets1951note,vinoles2016problemes,ziolkowski2001wave} for more details). 

We look for the following transmission problem:
\begin{equation} \label{eq:pbguide:nd2} 
\left\{
\begin{aligned}
 \Delta u^\e + (k^\e)^2 u^\e & = 0, &  \qquad &\text{in $(-\infty,0) \times \Gamma$,}\\
\Delta u^\I + (k^\I)^2 u^\I& = 0, &  \qquad &\text{in $(0,+\infty) \times \Gamma$,} \\
u^\I -u^\e &=  f,& \qquad &\text{on $\{0\} \times \Gamma$,}   \\
\partial_{\mf n}  u^\I  - \kappa \partial_{\mf n}  u^\e &= g,& \qquad &\text{on $\{0\} \times\Gamma$,} \\ 
\text{$u^\e$ verifies \eqref{eq:sommerfeld},}  & & \qquad & \text{when $x \to - \infty$,} \\
\text{$u^\I$ verifies \eqref{eq:sommerfeldInverse},}  & & \qquad & \text{when $x \to + \infty$,} \\
\text{+ BCs}  & & \qquad & \text{on $\R \times \partial \Gamma$.} 
\end{aligned}
\right.
\end{equation}

\subsubsection{Reduction to linear systems}

We now use the separation of variable \eqref{eq:series}. Taking into account the radiation conditions and the fact that we discard exponentially growing solutions, the solutions of \eqref{eq:pbguide:nd2} are given by
\begin{equation}\label{eq:decompositionnd2}
\begin{aligned}
& u^\e(x,\mf y) = \sum_{n \in \N}  u^\e_n e^{-i \beta^\e_n x}  \psi_n(\mf y),&\qquad &x < 0,\ \mf y \in \Gamma, \\
& u^\I(x,\mf y) = \sum_{n \in \N}  u^\I_n e^{-i \beta^\I_n x} \psi_n(\mf y) ,&\qquad& x>0, \ \mf y \in \Gamma, 
\end{aligned}
\end{equation}
where $u^\e_n$ and $u^\I_n$ are modal coefficients to determine. The transmission conditions of \eqref{eq:pbguide:nd2} write, using \eqref{eq:decompositionnd2}, as a countable family of $2 \times 2$ linear systems:
\begin{equation}\label{eq:systemesnd2}
\mathcal A_n \begin{bmatrix}
 u^\e_n \\ u^\I _n
\end{bmatrix} = 
\begin{bmatrix}
f_n \\ -i g_n
\end{bmatrix}, \quad \text{where} \qquad
\mathcal A_n := 
\begin{bmatrix}
  -1   &  1    \\
\kappa \beta^\e_n & - \beta^\I_n   
\end{bmatrix},
\end{equation}
for all $n \in \N$. Denote $\mathcal D_n :=   \beta^\I_n-\kappa \beta^\e_n$ the determinants associated to \eqref{eq:systemesnd2}. Contrary to Section \ref{sec:diskball}, these can actually vanish.

\begin{prop}\label{prop:detnd2}
For $\kappa \neq -1$ (standard case) or for $\kappa = -1$ and $k^\e \neq k^\I$ (critical case), the determinants $\mathcal D_n$ do not vanish except perhaps for a finite number of $n$. But for $\kappa = -1$ and $k^\e = k^\I$ (super-critical case), $\mathcal D_n $ vanishes for sufficiently large $n$.
\end{prop}

\begin{proof}
Let be $n \in \N$. Recall that we have excluded the cut-off wave numbers, $\beta^\I_n \neq 0$ and $\beta^\e_n \neq 0$ (or equivalently $\lambda_n \neq (k^\I)^2$ and $\lambda_n \neq (k^\e)^2$). We distinguish three cases:
\begin{enumerate}
\item $\lambda_n < \min((k^\I)^2, (k^\e)^2)$. Both $\beta^\I_n$ and $\beta^\e_n $ are positive numbers according to \eqref{eq:choixbeta}. Thus $\mathcal D_n =   \beta^\I_n-\kappa \beta^\e_n \neq 0$ since $\kappa < 0$. 
\item $\min((k^\I)^2, (k^\e)^2) < \lambda_n < \max((k^\I)^2, (k^\e)^2)$ (can only happens if $k^\I \neq k^\e$). Among $\beta^\I_n$ and $\beta^\e_n $ there are one non-zero real number and one non-zero imaginary number, so $\mathcal D_n \neq 0$.
\item $\lambda_n > \max((k^\I)^2, (k^\e)^2)$. Using \eqref{eq:choixbeta}, one has
\begin{equation}\label{eq:tmp1}
\mathcal D_n = 0 \Longleftrightarrow \sqrt{\lambda_{n}-(k^\I)^2} + \kappa \sqrt{\lambda_{n}-(k^\e)^2}   = 0.
\end{equation}
If $\kappa = -1$ and $k^\e = k^\I$, \eqref{eq:tmp1} holds. If $\kappa = -1$ and $k^\e \neq k^\I$, \eqref{eq:tmp1} does not hold. For the case $\kappa \neq -1$, \eqref{eq:tmp1} holds if and only if
\begin{equation}\label{eq:tmp2}
\lambda_{n} = \frac{\kappa^2  (k^\e)^2-(k^\I)^2}{\kappa^2-1}.
\end{equation}
That means that if such a $\lambda_n$ exists, it is unique, so there could be only a finite number of $n$ such that $\mathcal D_n = 0$ (the multiplicity of $\lambda_n$). 
\end{enumerate}
This ends the proof.
\end{proof}

When $\mathcal D_n$ vanishes, the corresponding system \eqref{eq:systemesnd2} has a non-empty kernel of dimension 1 spanned by $(1,1)^\mathrm{T}$. Consequently, the transmission problem \eqref{eq:pbguide:nd2} has a non-empty kernel (in the sense that there are non-trivial solutions of \eqref{eq:pbguide:nd2} for $(f,g) = (0,0)$). In the standard case, if it is non-empty, that is to say if \eqref{eq:tmp2} holds, its dimension is finite equal to the multiplicity of the corresponding $\lambda_n$. For the super-critical case, the kernel is always of infinite dimension because \eqref{eq:tmp1} holds as soon as $\lambda_n > \max((k^\I)^2, (k^\e)^2)$. In both cases, the kernel is spanned by the functions
\begin{equation}\label{eq:fonctionsNoyau}
G_n(x,\mf y) := \psi_n(\mf y) e^{- i \beta^\I_n \vert x \vert} = \psi_n(\mf y) e^{- i \kappa  \beta^\e_n \vert x \vert},\qquad  x \in \R,\quad \mf y \in \Gamma.
\end{equation} 
These functions are symmetric with respect to $x = 0$ and evanescent on each side on the interface, \textit{i.e.} they are localised near the interface $\Gamma$. Such solutions are called surface plasmons (\cite{maier2007plasmonics}). 

\begin{rmk}\label{rmk:limitR2}
Equation \eqref{eq:tmp1} is similar to the limit value of \eqref{eq:limitR}, where $\lambda_n$ plays the role of $\xi^2$. Furthermore, $\lambda_n > \max((k^\I)^2, (k^\e)^2)$ means that we are dealing with evanescent waves on both side of the interface, as mentioned in Remark \ref{rmk:limitR}.
\end{rmk}

\subsubsection{Asymptotic analysis}

Now we investigate the case where the determinant $\mathcal D_n$ does not vanish, (\textit{i.e.} not the super-critical case $\kappa= -1$ and $k^\e = k^\I$). As done in Section \ref{subsec:asymptoticBall}, we link the regularity of the solution $(u^\e,u^\I)$ to the decay of the modal coefficients  $(u^\e_n , u^\I _n)$ by introducing the space $ \mathfrak H^s(\Gamma)$, $s \geq 0$, defined as
\begin{equation}\label{eq:charasobolev0}
 \mathfrak H^s(\Gamma) := \left\{ u \in L^2(\Gamma)\ \colon  \sum_{n \in \N} (1+\lambda_n)^s \vert u_n \vert^2 <  +\infty\right \},
\end{equation}
where $u_n := \langle u,\psi_n \rangle_{L^2(\Gamma)}$. 
This definition can be extended by duality to negative exponents:
\begin{equation}\label{eq:charasobolev00}
 \mathfrak H^{-s}(\Gamma) := \left\{ \phi \in \mathcal C^\infty(\Gamma)^*  \colon  \sum_{n \in \N} (1+\lambda_n)^{-s} \vert \phi_n \vert^2 <  +\infty\right \},
\end{equation}
where $\phi_n := \phi(\overline{\psi_n})$. One can characterise these spaces using the interpolation theory between Hilbert spaces (see \cite{lions2012non,huet1976decomposition}). This characterisation crucially depends on the dimension but also on the boundary conditions imposed on $\partial \Gamma$. For instance (see \cite{hazard2008improved}), if $\Gamma = (0,1) \subset \R$ with homogeneous Neumann conditions $u^\prime(0) = u^\prime(1) = 0$, then 
\begin{equation}
\mathfrak H^s(\Gamma) = \begin{cases}
H^s(\Gamma) & \text{if $0 \leq s < 3/2$,} \\
\{ u \in H^s(\Gamma) \colon u^\prime(0) = u^\prime(1) = 0\}  & \text{if $3/2 \leq s < 7/2$,} \\
\{ u \in H^s(\Gamma) \colon u^\prime(0) = u^\prime(1) = u^{\prime\prime\prime}(0) = u^{\prime\prime\prime}(1)  = 0\}  & \text{if $7/2 \leq s < 11/2$,}
\end{cases}
\end{equation}  
and so on: the boundary condition $u^{2n-1}(0) = u^{2n-1}(1)  = 0$ appears as soon as it makes sense, \textit{i.e.} as soon as $s \geq (4n-1)/2$. In other words, the convergence of the series in \eqref{eq:charasobolev0} depends not only on the Sobolev regularity of $u$ but also on its behaviour on $\partial \Gamma$.  In the following, we will not try to characterise $\mathfrak H^s(\Gamma)$ since all the analysis remains the same for all dimension $d \geq 2$ and for any boundary conditions that makes $\Delta_{\mf y}$ self-adjoint with compact resolvent. Instead we stick with the spaces $\mathfrak H^s(\Gamma)$ and just focus on the Sobolev regularity through the asymptotic behaviour of $u_n$.

We now follow the steps of Section \ref{subsec:asymptoticBall}. Solving \eqref{eq:systemesnd2} leads to
\begin{equation}\label{eq:solutionnd2}
\mathcal A_n \begin{bmatrix}
 u^\e_n \\ u^\I _n
\end{bmatrix} = (\mathcal A_n)^{-1} 
\begin{bmatrix}
f_n \\ -i g_n
\end{bmatrix} = \frac{-1}{\mathcal D_n} 
\begin{bmatrix}
 \beta^\I_n &   1 \\
   \kappa \beta^\e_n   &  1
\end{bmatrix} \begin{bmatrix}
f_n \\ -i g_n
\end{bmatrix}.
\end{equation}

\begin{prop}
One has
\begin{equation}\label{eq:asympDet}
\mathcal D_n \underset{n \to +\infty}{\sim} \left\{
\begin{aligned}
 & i \sqrt{\lambda_n} (1+\kappa ) & \qquad &\text{if $\kappa  \neq -1$ (standard case)}\\
 & i   \frac{(k^\e)^2- (k^\I)^2}{2\sqrt{\lambda_n}}  & \qquad& \text{if $\kappa  =-1$ and $k^\e \neq k^\I$ (critical case).}
\end{aligned} \right.
\end{equation}
\end{prop}

\begin{proof}
We can suppose that the $\lambda_n$ are large enough such that (see \eqref{eq:choixbeta})
\begin{equation}\label{eq:asympBeta}
\begin{aligned}
& \beta_n^\I = -i \sqrt{\lambda_n  - (k^\I)^2} =  -i \sqrt{\lambda_n} \left[  \sum_{j=0}^N (-1)^j \binom{1/2}{j} \left( \frac{(k^\I)^2 }{\lambda_n} \right)^j + \mathcal O\left( \frac{1}{(\lambda_n)^{N+1}} \right) \right], \\
& \beta_n^\e = i \sqrt{\lambda_n  - (k^\e)^2} =i \sqrt{\lambda_n} \left[  \sum_{j=0}^N (-1)^j \binom{1/2}{j} \left( \frac{(k^\e)^2 }{\lambda_n} \right)^j + \mathcal O\left( \frac{1}{(\lambda_n)^{N+1}} \right) \right],
\end{aligned}
\end{equation}
Then we get
\begin{equation}
\mathcal D_n =   \kappa  \beta^\e_n - \beta^\I_n = i \sqrt{\lambda_n} \left[ (1+\kappa ) - \frac{(k^\I)^2+\kappa  (k^\e)^2}{2\lambda_n} + \mathcal O\left(\frac{1}{\lambda_n^2} \right) \right].
\end{equation}
It is now easy to conclude: if $\kappa  \neq -1$, the first term in the asymptotic does not vanish and we get the desired result. Now if $\kappa  = -1$ and $k^\I\neq k^\e$ , this first term vanishes but the not the second one, and the result follows.
\end{proof}

We can now give the asymptotics of $u^\I_n$ and $u^\e_n$:
\begin{prop}
One has
\begin{equation}
\begin{bmatrix}
 u^\I_n \\ u^\e _n
\end{bmatrix} \underset{n \to +\infty}{\sim}  \left\{
\begin{aligned}
 &  
\frac{1}{1+\kappa} \mathcal M_{n,\kappa} (0)  \begin{bmatrix}
f_n \\ -i g_n
\end{bmatrix}
 & \qquad &\text{if $\kappa  \neq -1$ }\\
 &  
 \frac{2}{(k^\e)^2 - (k^\I)^2} \mathcal M_{n,-1}(2)  \begin{bmatrix}
f_n \\ -i g_n
\end{bmatrix}
 & \qquad& \text{if $\kappa  =-1$ and $k^\e \neq k^\I$} \\
\end{aligned} \right.
\end{equation}
where  $\mathcal M_{n,\kappa}(p)$ is the matrix
\begin{equation}
\mathcal M_{n,\kappa}(p) := \begin{bmatrix}
\kappa (\lambda_n)^{p/2} &   (\lambda_n)^{(p-1)/2} \\ 
- (\lambda_n)^{p/2} &   (\lambda_n)^{(p-1)/2}
\end{bmatrix}.
\end{equation}
\end{prop}
\begin{proof}
Combine \eqref{eq:solutionnd2}, \eqref{eq:asympDet} and \eqref{eq:asympBeta}.
\end{proof}

\subsubsection{Conclusion}
\label{sec:conclusion2}

As we did in Section \ref{sec:diskball}, by gathering the results and using the characterisations \eqref{eq:charasobolev0} and \eqref{eq:charasobolev00}, we can conclude:

\begin{thm}\label{thm:guide2} 
Let $s > 0$ and consider the transmission problem \eqref{eq:pbguide:nd2}. Then
\begin{itemize}
\item if $\kappa  \neq -1$ (standard case), for $(f,g) \in \mathfrak H^s(\Gamma) \times \mathfrak H^{s-1}(\Gamma)$, \eqref{eq:pbguide:nd2} has a unique solution $(u^\I\vert_\Gamma,u^\e\vert_\Gamma) \in \mathfrak H^s(\Gamma)^2$ (no order of regularity lost), except in the exceptional situation where \eqref{eq:tmp1} holds; it has a kernel of finite dimension equal to the multiplicity of the corresponding $\lambda_n$ spanned by the evanescent functions \eqref{eq:fonctionsNoyau};
\item if $\kappa  = -1$ and and $k^\e \neq k^\I$ (critical case), for $(f,g) \in \mathfrak H^{s+2}(\Gamma) \times \mathfrak H^{s+1}(\Gamma)$ , \eqref{eq:pbguide:nd2} has a unique solution $(u^\I\vert_\Gamma,u^\e\vert_\Gamma) \in \mathfrak H^s(\Gamma)^2$ (2 orders of regularity lost);
\item if $\kappa  = -1$ and $k^\e = k^\I$ (super-critical case), \eqref{eq:pbguide:nd2} has a kernel of infinite dimension spanned by the evanescent functions \eqref{eq:fonctionsNoyau} for all $n$ such that $\lambda_n > \max((k^\I)^2,(k^\e)^2)$.
\end{itemize}
\end{thm}

These results are summarised in Table \ref{table:nd2}. We have a strongly ill-posed problem for the super-critical case $\kappa  = -1$ and $k^\e = k^\I$ (for instance it escapes the Fredholm framework). We can also reinterpret the results in terms of volume source  as we done at the end of Section \ref{sec:conclusion} for the standard and the critical cases.

\begin{table} 
\centering
\begin{tabular}{|c|c|c|}
\hline
 & $\kappa \neq -1$ & $ \kappa = -1$ \\
 \hline
 $k^\e \neq k^\I$ & 0 & 2\\
 \hline
 $k^\e = k^\I$ & 0 & kernel of infinite dimension \\
 \hline
\end{tabular}
\
\caption{\label{table:nd2} orders of regularity lost solving \eqref{eq:pbguide:nd2}.}
\end{table}

\begin{rmk}\label{rmk:cutoff}
As claimed before, excluding cut-off wave numbers does not change the conclusion of the Theorem \ref{thm:guide2}. Indeed, it would eventually just add a finite numbers of elements to the kernel.
\end{rmk}

\subsection{A case where the negative material is bounded}

The previous situation is in some sense the ``worst'' we can encounter. Let us take a look to a case where $\Omega^\I$ is bounded. For instance, consider $\Omega^\I = (0,2L)$, with $L>0$, so that $\Omega^\e = (-\infty,-0) \cup (2L,+\infty)$. For this problem, it is more convenient to decompose the solution as the sum of two functions that are respectively symmetric and skew-symmetric (with respect to $x = L$). Doing so, our problem boils down to the study of two problems with $\Omega^\I = (0,L)$ and $\Omega^\e = (-\infty,0)$ (see Figure \ref{fig:pbguide:nd}), with the addition of an homogeneous Dirichlet condition (resp. homogeneous Neumann condition) at $x = L$ corresponding to the skew-symmetric part (resp. symmetric part). In the following, we focus on the Dirichlet case, however all the conclusions still hold for the Neumann case, thus for the original problem  $\Omega^\I = (0,2L)$.

The transmission problem we look for is (see Figure \ref{fig:pbguide:nd})
 \begin{equation} \label{eq:pbguide:nd} 
\left\{
\begin{aligned}
\Delta u^\e + (k^\e)^2 u^\e& = 0, &  \qquad &\text{in $(-\infty,0) \times \Gamma$,}\\
\Delta u^\I + (k^\I)^2 u^\I & = 0, &  \qquad &\text{in $(0,L) \times \Gamma$,} \\
u^\I -u^\e &=  f,& \qquad &\text{on $\{0\} \times \Gamma$,}   \\
\partial_{\mf n}  u^\I  - \kappa  \partial_{\mf n}  u^\e &= g,& \qquad &\text{on $\{0\} \times\Gamma$,} \\ 
u^\I & =0,& \qquad &\text{on $\{L\} \times \Gamma$,}   \\
\text{$u^\e$ verifies \eqref{eq:sommerfeld},}  & & \qquad & \text{when $\vert \mf x \vert \to - \infty$,}\\
\text{+ BCs} & & \qquad & \text{on $\R \times \Gamma$.}
\end{aligned}
\right.
\end{equation}

\begin{figure}[!htbp]   
\centering
\includegraphics[width=0.75\textwidth]{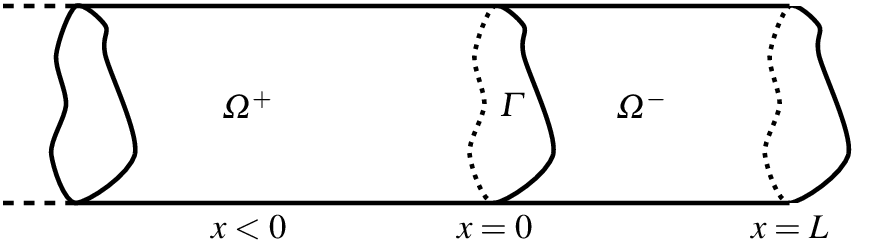}
\caption{\label{fig:pbguide:nd}Geometry of the problem \eqref{eq:pbguide:nd}.}
\end{figure}

\subsubsection{Reduction to linear systems}

Following the same steps as before, we look for solutions under the form
\begin{equation}\label{eq:decompositionnd}
\begin{aligned}
& u^\e(x,\mf y) = \sum_{ n \in \N} u^\e_n  e^{-i \beta^\e_n x} \psi_n(\mf y) ,&\qquad &x < 0,\ \mf y \in \Gamma, \\
& u^\I(x,\mf y) = \sum_{ n \in \N} \left( u^\I_{n,+} e^{i \beta^\I_n x} + u^\I_{n,-} e^{-i \beta^\I_n x} \right) \psi_n(\mf y) ,&\qquad& x \in (0,L), \ \mf y \in \Gamma,
\end{aligned}
\end{equation}
where $\beta^\I_n$ et $\beta^\e_n$ are defined by \eqref{eq:choixbeta}. The transmissions conditions of \eqref{eq:pbguide:nd} and the Dirichlet boundary condition at $x = L$ leads to a countable family of $3 \times 3$ linear systems:
\begin{equation}\label{eq:systnd} 
\mathcal A_n
\begin{bmatrix}
u^\e_{n} \\ u^\I_{n,+} \\ u^\I_{n,-} 
\end{bmatrix}
=
\begin{bmatrix}
f_n \\ g_n \\ 0
\end{bmatrix}, \quad \text{where} \quad 
\mathcal A_n :=
\begin{bmatrix}
-1 & 1  & 1   \\
\kappa  \beta^\e_n  &  \beta^\I_n &- \beta^\I_n  \\
 0 & e^{i \beta^\I_n L}  &   e^{- i \beta^\I_n L}   \\
\end{bmatrix}.
\end{equation}
The determinants $\mathcal D_n$ associated to \eqref{eq:systnd} are
\begin{equation}\label{eq:defdeterminants}
\mathcal D_n := -2  \beta^\I_n \cos( \beta^\I_n L)+ 2i \kappa \beta^\e_n \sin( \beta^\I_n L).
\end{equation}

\begin{prop}
For all $n \in \N$, one has $\mathcal D_n \neq 0$ except perhaps for a finite number of $n$.
\end{prop}

\begin{proof}
Since we excluded cut-off wave numbers, $\beta^\I_n \neq 0$ and $\beta^\e_n \neq 0$. Notice that $\cos( \beta^\I_n L)$ and $\sin( \beta^\I_n L)$ cannot vanish simultaneously. We distinguish 3 cases:
\begin{enumerate}
\item $\lambda_n < \min((k^\I)^2, (k^\e)^2)$. Both $\beta^\I_n$ and $\beta^\e_n $ are real according to \eqref{eq:choixbeta}. Thus $\mathcal D_n \neq 0$.
\item $\min((k^\I)^2, (k^\e)^2) < \lambda_n < \max((k^\I)^2, (k^\e)^2) (k^\I)^2$ (can only happens if $k^\I \neq k^\e$). There are two possibilities:
\begin{itemize}
\item $k^\e > k^\I$, so $(k^\I)^2 <  \lambda_n < (k^\e)^2$. According to \eqref{eq:choixbeta}, $\beta^\I_n$ is purely imaginary whereas $\beta_n^\e$ is real, so $ -2 \beta^\I_n \cos( \beta^\I_n L)$ is purely imaginary and $2i \kappa \beta^\e_n \sin( \beta^\I_n L)$ is real, thus $\mathcal D_n \neq 0$.
\item $k^\e < k^\I$, so $(k^\e)^2 <  \lambda_n < (k^\I)^2$. According to \eqref{eq:choixbeta} and \eqref{eq:defdeterminants}, the equation $\mathcal D_n = 0$ becomes
\begin{equation}\label{eq:tmp3}
\sqrt{(k^\I)^2-\lambda_n} \cos \left( \sqrt{(k^\I)^2-\lambda_n}  \right) + 2  \kappa \sqrt{\lambda_n-(k^\e)^2} \sin\left( \sqrt{(k^\I)^2-\lambda_n}  \right) = 0.
\end{equation}
Seen as an equation of unknown $\lambda_n$, \eqref{eq:tmp3} could only have a finite number of solutions in $((k^\e)^2 , (k^\I)^2)$ because its left hand-side defines a non-zero holomorphic function on the ball centred in $((k^\I)^2+(k^\e)^2)/2$ of radius $(k^\e)^2 - (k^\I)^2$. 
\end{itemize}
\item $\lambda_n > \max((k^\I)^2, (k^\e)^2)$. According to \eqref{eq:choixbeta} and \eqref{eq:defdeterminants}, the equation $\mathcal D_n = 0$ becomes
\begin{equation}\label{eq:tmp4}
\sqrt{\lambda_n-(k^\I)^2} \cosh \left( \sqrt{\lambda_n-(k^\I)^2}   \right) +  \kappa \sqrt{\lambda_n-(k^\e)^2} \sinh\left( \sqrt{\lambda_n-(k^\I)^2}   \right) = 0.
\end{equation}
Again, seen as an equation in $\lambda_n$, \eqref{eq:tmp3} could only have a finite number of solutions in $I=(\max((k^\I)^2, (k^\e)^2),+\infty)$. In each bounded subset of $I$, it could have only a finite number of zero (again because the left hand-side of \eqref{eq:tmp3} defines a non-zero holomorphic function on the half-space $\{z \in\C \ \colon \Im z > \max((k^\I)^2, (k^\e)^2 \}$) and for $\lambda_n$ large enough $\mathcal D_n$ does not vanish (see the asymptotics of Proposition \ref{prop:asymptotiqueDn}).
\end{enumerate}
This ends the proof.
\end{proof}

When $\mathcal D_n$ vanishes, the corresponding system \eqref{eq:systnd} has a non-empty kernel of dimension 1 spanned by $(2i \sin(\beta_n^\I L),-e^{i \beta_n^\I L},e^{i \beta_n^\I L} )^\mathrm{T}$. Consequently, the transmission problem \eqref{eq:pbguide:nd} has a non-empty kernel of finite dimension, spanned by
\begin{equation}\label{eq:fonctionsNoyau2}
G_n(x,\mf y) := \begin{cases} 2i \sin(\beta_n^\I L) \psi_n(\mf y) e^{-i \beta^\e_n \vert x \vert} & \text{for $x < 0$, $\mf y \in \Gamma$,} \\
 2 i \psi_n(\mf y) \sin( \beta_n^\I (L-x)) & \text{for $x > 0$, $\mf y \in \Gamma$.}
 \end{cases}
 \end{equation} 
When \eqref{eq:tmp3} holds, it means that $(k^\e)^2 <  \lambda_n < (k^\I)^2$ so $\beta_n^\e$ is real whereas $\beta_n^\I$ is purely imaginary. Consequently, the corresponding $G_n$ are evanescent in $\Omega^\e$. Thus these functions correspond to the so-called trapped modes (in the sense that $G_n$ is localised in the bounded domain $\Omega^\I$). Notice that they could exist without change of sign: \eqref{eq:tmp3} can hold even when $\kappa >0$ (see \cite{linton2007embedded} for more details about trapped modes). When \eqref{eq:tmp4} holds, since $\lambda_n > \max((k^\I)^2, (k^\e)^2)$, both $\beta_n^\I$ and $\beta_n^\e$ are purely imaginary, thus the corresponding $G_n$ is evanescent of each side of the interface (surface plasmons). Such solution cannot exist when $\kappa >0$, \textit{i.e.} without changes of sign.

\subsubsection{Asymptotic analysis}

Following the same steps as in the previous section, when $\mathcal D_n \neq 0$ we can first solve \eqref{eq:systnd}:
\begin{equation}\label{eq:solutionnd}
\begin{bmatrix}
u^\e_{n} \\ u^\I_{n,+} \\ u^\I_{n,-} 
\end{bmatrix}= (\mathcal A_n)^{-1} 
\begin{bmatrix}
f_n \\ g_n \\ 0
\end{bmatrix}, \quad \text{where} \quad 
 (\mathcal A_n)^{-1} 
 =  \frac{1}{\mathcal D_n}
 \begin{bmatrix}
 2 \beta^\I_n \cos(\beta_n^\I L) &  2 i  \cos(\beta_n^\I L) & - 2 \beta^\I_n \\
 - \kappa \beta^\e_n e^{-i \beta^\I_n L} & -e^{-i \beta^\I_n L} & \kappa \beta^\e_n- \beta^\I_n \\
 \kappa \beta^\e_n e^{i \beta^\I_n L} & e^{i \beta^\I_n L} & -\kappa \beta^\e_n- \beta^\I_n  
 \end{bmatrix}.
\end{equation}

We now compute the asymptotic of $\mathcal D_n$:

\begin{prop}\label{prop:asymptotiqueDn}
One has
\begin{equation}
\mathcal D_n \underset{n \to +\infty}{\sim} \left\{
\begin{aligned}
 i (1+\kappa) \sqrt{\lambda_n} e^{L \sqrt{\lambda_n}} & \qquad \text{if $\kappa  \neq -1$},\\
 i \frac{(k^\e)^2-(k^\I)^2}{2} \frac{1}{\lambda_n} e^{L \sqrt{\lambda_n}} & \qquad \text{if $\kappa  =-1$ and $k^\e \neq k^\I$}, \\
2 i \sqrt{\lambda_n} e^{- L \sqrt{\lambda_n}}  & \qquad \text{if $\kappa  =-1$ and $k^\e = k^\I$} .
\end{aligned} \right.
\end{equation}
\end{prop}

\begin{proof}
The first two cases are obtained exactly like the ones of Proposition \ref{eq:asympDet2d}. For $\kappa  =-1$ and $k^\e = k^\I$, notice that $\mathcal D_n =  2 \beta^\e_n e^{i \beta^\e_n L} = - 2 \beta^\I_n e^{-i \beta^\I_n L} $ (for $n$ large enough). The result in this case is thus straightforward.
\end{proof}

Finally, one gets the asymptotics of the modal coefficients:

\begin{prop}
One has
\begin{equation}\label{eq:tmptmptmp}
\begin{bmatrix}
u^\e_{n} \\ u^\I_{n,+} \\ u^\I_{n,-} 
\end{bmatrix} \underset{n \to +\infty}{\sim}
  \left\{
\begin{aligned}
 &  
\frac{1}{1+\kappa} \mathcal M_{n,\kappa,1} (0)  \begin{bmatrix}
f_n \\ g_n \\ 0
\end{bmatrix},
 & \qquad &\text{if $\kappa  \neq -1$,}\\
 &  
 \frac{2}{(k^\e)^2 - (k^\I)^2} \mathcal M_{n,-1,2}(2)  \begin{bmatrix}
f_n \\ g_n \\ 0
\end{bmatrix},
 & \qquad & \text{if $\kappa  =-1$ and $k^\e \neq k^\I$,} \\
& \frac{1}{2} e^{2 L \sqrt{\lambda_n}} \mathcal M_{n,-1,3} (0)  \begin{bmatrix}
f_n \\ g_n \\ 0
\end{bmatrix},
& \qquad & \text{if $\kappa  =-1$ and $k^\e = k^\I$,} 
\end{aligned} \right.
\end{equation}
where $\mathcal M_{n,\kappa,j}(p)$ is the matrix
\begin{equation}
\mathcal M_{n,\kappa,j}(p) := \begin{bmatrix}
-1 &   -i (\lambda_n)^{(p-1)/2} & 2 e^{-L \sqrt{\lambda_n}}\\ 
- \kappa (\lambda_n)^{p/2} e^{-2 L \sqrt{\lambda_n}} &   i (\lambda_n)^{(p-1)/2}  e^{-2 L \sqrt{\lambda_n}} & \delta_j e^{-2 L \sqrt{\lambda_n}} \\
\kappa (\lambda_n)^{p/2} & -i (\lambda_n)^{(p-1)/2}  & (1-\kappa) e^{-L \sqrt{\lambda_n}} 
\end{bmatrix},
\end{equation}
with
\begin{equation}
\delta_j = \begin{cases} (1+\kappa) & \text{if $j = 1$ ($\kappa \neq -1$),} \\
\frac{(k^\e)^2 - (k^\I)^2}{2}  & \text{if $j = 2$ ($\kappa = -1$ and $k^\e \neq k^\I$),} \\ 
0 & \text{if $j = 3$ ($\kappa = -1$ and $k^\e = k^\I$).}
\end{cases}
\end{equation}
\end{prop}

\subsubsection{Conclusion}
\label{subsubsec:conclusion}

Notice that, in the super-critical case $\kappa = -1$ and $k^\e=k^\I$, one gets a factor $e^{2 L \sqrt{\lambda_n}}$ in front of $\mathcal M_{n,-1,3} (0)$. Thus we need to introduce the following weighted spaces analogous to \eqref{eq:charasobolev0} for $s>0$ and $L \geq 0$:
\begin{equation}\label{eq:charasobolevGL}
 \mathfrak G_L^s(\phi) := \left\{ u \in L^2(\phi)\colon  \sum_{n \in \N} e^{2L \sqrt{\lambda_n}} (1+\lambda_n)^s \vert u_n \vert^2 <  +\infty\right \}.
\end{equation}
Notice that we have the following inclusions for $s^\prime \geq s > 0$ and $L^\prime \leq L \leq 0$:
\begin{equation}\label{eq:inclusions}
\mathfrak G_L^{s^\prime}(\Gamma) \subset \mathfrak G_L^{s}(\Gamma) \qquad \text{and} \qquad 
\mathfrak G_{L^\prime}^{s}(\Gamma) \subset \mathfrak G_L^{s}(\Gamma) .
\end{equation}
The condition $\sum_{n \in \N} e^{2L \sqrt{\lambda_n}} (1+\lambda^s_n) \vert u_n \vert^2 <  +\infty$ is restrictive because it imposes an exponential decay of the modal coefficients of the functions belonging to $\mathfrak G_L^s(\Gamma)$.
We can extend the definition of $ \mathfrak G_L^s(\Gamma) $ by duality to negative exponents:
\begin{equation}\label{eq:charasobolevGL2}
 \mathfrak G^{-s}(\Gamma) := \left\{ \phi \in \mathcal C^\infty(\Gamma)^* \colon  \sum_{n \in \N} e^{2L \sqrt{\lambda_n}} (1+\lambda_n)^{-s} \vert \phi_n \vert^2 <  +\infty\right \}.
\end{equation}

It is now possible to conclude:
\begin{thm}\label{thm:guide}
Let $s> 0$ and consider the transmission problem \eqref{eq:pbguide:nd}: 
\begin{itemize}
\item if $\kappa  \neq -1$ (standard case), for $(f,g) \in \mathfrak H^s(\Gamma) \times \mathfrak H^{s-1}(\Gamma)$, \eqref{eq:pbguide:nd} admits a unique solution $(u^\I\vert_\Gamma,u^\e\vert_\Gamma) \in \mathfrak H^s(\Gamma)^2$ (no order of regularity lost);
\item if $\kappa  = -1$ and and $k^\e \neq k^\I$ (critical case), for $(f,g) \in \mathfrak H^{s+2}(\Gamma) \times \mathfrak H^{s+1}(\Gamma)$, \eqref{eq:pbguide:nd} has a unique solution $(u^\I\vert_\Gamma,u^\e\vert_\Gamma) \in \mathfrak H^s(\Gamma)^2$ (2 orders of regularity lost);
\item if $\kappa  = -1$ and $k^\e = k^\I$ (super-critical case), for $(f,g) \in \mathfrak G_L^s(\Gamma) \times \mathfrak G_L^{s-1}(\Gamma)$,  \eqref{eq:pbguide:nd} has a unique solution $(u^\I\vert_\Gamma,u^\e\vert_\Gamma) \in \mathfrak H^s(\Gamma)^2$ (``infinite" order of regularity lost);
\end{itemize}
except in the exceptional situations when \eqref{eq:tmp3} or \eqref{eq:tmp4} holds. In this case, it has a kernel of finite dimension spanned by the evanescent functions \eqref{eq:fonctionsNoyau2} (trapped modes or evanescent modes).
\end{thm}
Notice that Remark \ref{rmk:cutoff} still holds in this situation. These results are summarised in Table \ref{table:nd}. We can also reinterpret the results in term of volume source  as we done at the end of Sections \ref{sec:conclusion} and \ref{sec:conclusion2} for the standard and the critical cases. 

\begin{table} 
\centering
\begin{tabular}{|c|c|c|}
\hline
 & $\kappa \neq -1$ & $ \kappa = -1$ \\
 \hline
 $k^\e \neq k^\I$ & 0 & 2\\
 \hline
 $k^\e = k^\I$ & 0 & $\infty$ \\
 \hline
\end{tabular}
\
\caption{\label{table:nd} orders of regularity lost solving \eqref{eq:pbguide:nd}.}
\end{table}

For the super-critical case, the concluding observation of Remark \ref{rmk:source} when the source $F$ is compactly supported in $\Omega^\e$ does not hold any more. Indeed, one can have $f \in \mathcal C^\infty(\Gamma)$ without having $f \in \mathfrak G_L^{s}(\Gamma)$. Denote by $d(S,\Gamma)$ the Hausdorff distance between the support $S$ of $F$ and the interface $\Gamma$, and denote by $h$ the trace of $u^\mathrm{inc}$ on $\Gamma_F := \{ -d(S,\Gamma) \} \times \Gamma$ located at $x = -d(S,\Gamma)$. Then $u^\mathrm{inc}$ is the outgoing solution of the problem $ \Delta u^\mathrm{inc} + (k^\e)^2 u^\mathrm{inc}= 0$ on $(-d(S,\Gamma),+\infty) \times \Gamma$ with the condition $u^\mathrm{inc} = h$ on $\Gamma_F$. It can be given explicitly:
\begin{equation}
u^\mathrm{inc}(x,\mf y) = \sum_{n \in \N}  h_n  e^{i \beta^\e_n (x+d(S,\Gamma))} \psi_n(\mf y).
\end{equation}
where $h_n = \langle h,\psi_n \rangle_{L^2(\Gamma)}$. It means that the modal coefficients $f_n$ of $f = u^\mathrm{inc} \vert_\Gamma$ satisfy $f_n = h_n e^{i \beta^\e_n d(S,\Gamma)}$ so using \eqref{eq:choixbeta} one gets
\begin{equation}\label{eq:tmpFinal}
e^{2L \sqrt{\lambda_n}} f_n  =   e^{2L \sqrt{\lambda_n}} e^{i \beta^\e_n d(S,\Gamma)} h_n \underset{n \to +\infty}{\sim}  e^{(2L-d(S,\Gamma))\sqrt{\lambda_n}}h_n.
\end{equation}
Suppose now that $h \in \mathfrak H^{s}(\Gamma)$, $s>0$. If $d(S,\Gamma) \geq 2L$, \eqref{eq:tmpFinal} combined with \eqref{eq:charasobolev0} and  \eqref{eq:charasobolevGL} gives $f_n \in \mathfrak G_L^{s}(\Gamma)$. In a similar way, one has also $g \in \mathfrak G_L^{s-1}(\Gamma)$). Thus, using Theorem \ref{thm:guide}, \eqref{eq:pbguide:nd} is well-posed and we get $(u^\I\vert_\Gamma,u^\e\vert_\Gamma) \in \mathfrak H^{s}(\Gamma)^2$. Now if $d(S,\Gamma) < 2L$, coming back to \eqref{eq:tmptmptmp} and using \eqref{eq:tmpFinal}, one can see that the modal coefficients $u^\e_{n}$, $u^\I_{n,+}$ and $u^\I_{n,-} $ are growing exponentially. This means that the corresponding $u^\e\vert_\Gamma$ and $u^\I\vert_\Gamma$ are not even distributions on $\Gamma$ of finite order. In other words, the condition $(f,g) \in \mathfrak G_L^s(\Gamma) \times \mathfrak G_L^{s-1}(\Gamma)$ in Theorem \ref{thm:guide} is truly restrictive since it imposes that the source $F$ must be supported far away from the interface $\Gamma$, at a distance at least $2L$.

\section{Discussion and prospects}
\label{ref:comments}

Even if our analysis was able to finely characterise the loses of regularity of the considered problems, it is inevitably limited to particular geometries for which separation of variables is possible. For more general domains, when $\Omega^\I$ is bounded, only partial results have been proved, for $d\geq 3$ and when $\Omega^\I$ is strictly convex in \cite{ola1995remarks} and \cite{nguyen2016limiting}. This approach can also handle the case $d=2 $ with $k^\e \neq k^\I$ (critical case) but seems to fail irremediably when $\Omega^\I$ is not strictly convex for $d\geq 3$ and when $k^\e = k^\I$ (super-critical case) for $d=2$. It appears that we need some new idea to tackle these two cases. 

Another interesting problem is to deal with the full Maxwell equations (for $d=3$) instead of the Helmholtz equation. When $\kappa = -1$, very few has been done for these equations when involving sign-changing coefficients, even for smooth interfaces or simple geometries. Let us mention the paper \cite{dhia2014t} where the authors use results on scalar problems with sign-changing coefficients to deduce results on the full Maxwell equations. This approach could be certainly used in other situations.  

To conclude, let us mention that tremendous difficulties appear when the interface is not smooth any more (when it has corners for instance). In this case, in order to have well-posedness in $H^1$, the contrasts must lie outside an interval called the critical interval that contains $\{-1\}$. If they do not (but are different of $-1$), solutions exhibit strongly oscillating behaviour near the corners (\cite{dhia2012t,bonnet2013radiation}). One has to add some radiation conditions at the corners and to change the functional framework to recover well-posedness (as we did in this paper for the critical and super-critical cases). It is now well understood for $d=2$ but, as mentioned before, for $d=3$ (Maxwell equations) there is a lot to investigate, due to the fact that the geometries in 3d can much more complex than in 2d (it can have corners, edges, conical points, etc.). Finally, to our knowledge, the case where the contrasts are equal to $-1$ when the interface is not smooth has never been investigated.

\appendix

\section{Appendix: justification of the radiation conditions for negative materials}
\label{sec:radiationcondition}

In this Section, we justify that the ``correct'' (\textit{i.e.} physically relevant) radiation condition in media for which the coefficients are negative is \eqref{eq:sommerfeldInverse} instead of \eqref{eq:sommerfeld}.

For simplicity, we restrict ourselves to the dimension $d=1$, but one can proceed similary for higher dimensions. The method consists in using the limiting absorption principle (\cite{eidus1963principle}). It characterised the ``correct'' solution as the limit, when the dissipation tends to 0, of the unique solution of the same problem when the medium is absorbing, \textit{i.e.} the coefficients have a non-zero imaginary part. 

More precisely, consider the Helmholtz equation $u^{\prime \prime} + k^2 u =0$ where $k := \omega  \sqrt{\eps \mu}$ is a fixed wave number with $(\eps,\mu) := (\eps^\e,\mu^\e)$ or $(\eps,\mu) := (\eps^\I,\mu^\I)$. We want to determine what is the radiation condition to impose when $x$ tends to $+\infty$ (the case $-\infty$ is analogous). Suppose that the background medium is slightly absorbing, so that one has a permittivity $\eps_\eta$ and a permeability $\mu_\gamma$ which are now complex numbers:
\begin{equation}
\eps_\eta:= \eps + i \eta\qquad \text{and} \qquad \mu_\gamma := \mu + i \gamma,
\end{equation}
where $\eta>0$ and $\gamma > 0$ represent the absorption terms (see Remark \ref{rmk:absorptionSign}). We now define the corresponding wave number $k_{\eta,\gamma}$ such that $k_{\eta,\gamma} = (\omega^2  \eps_\eta \mu_\gamma)^{1/2}$, where we choose for the square root the ones which has $\R_+$ for the branch cut (this choice is arbitrary, another choice would lead to the same results):
\begin{equation}\label{defi:sqrt} 
z^{1/2} := \sqrt{\vert z \vert }  e^{i \arg z  /2},\qquad z \in \C\setminus \R^*_+,\ \arg z \in (0,2\pi).
\end{equation}
The solutions $u_{\eta,\gamma}$ of the Helmholtz equation $u_{\eta,\gamma}^{\prime \prime} + (k_{\eta,\gamma} )^2 u_{\eta,\gamma} =0$ are given by
\begin{equation}
u_{\eta,\gamma}(x) = A e^{i k_{\eta,\gamma}  x} + B e^{-i k_{\eta,\gamma}  x},
\end{equation}
for some constants $A$ and $B$. Since the imaginary part of $k_{\eta,\gamma}$ is always positive (see \eqref{defi:sqrt}), $e^{i k_{\eta,\gamma} x} $ is bounded when $x$ tend to $+\infty$ but $e^{-i k_{\eta,\gamma} x}$ is not. So one must impose $B=0$, and doing so one gets $u_{\eta,\gamma}(x) = A e^{i k_{\eta,\gamma}  x}$.  Moreover, using $\eps_\eta \mu_\gamma = (\eps \mu - \eta \gamma) + i(\eps \gamma + \mu \eta)$, the imaginary part of $\eps_\eta \mu_\gamma$ is positive when $(\eps,\mu) = (\eps^\e,\mu^\e)$ and negative when $(\eps,\mu) = (\eps^\I,\mu^\I)$. We obtain, according to \eqref{defi:sqrt}, that
\begin{equation}
\myRe k_{\eta,\gamma} > 0 \quad  \text{if $(\eps,\mu) = (\eps^\e,\mu^\e)$ } \qquad \text{and} \qquad  \myRe k_{\eta,\gamma} < 0 \quad\text{if $(\eps,\mu) = (\eps^\I,\mu^\I)$. } 
\end{equation}
Thus, we get
\begin{equation}\label{eq:choixk} 
\lim_{ \eta ,\gamma \to 0} k_{\eta,\gamma} =\begin{cases}  k^\e &  \text{if $(\eps,\mu) = (\eps^\e,\mu^\e)$, }  \\ -k^\I \ & \text{if $(\eps,\mu) = (\eps^\I,\mu^\I)$, }  \end{cases}
\end{equation}
and this implies
\begin{equation}
\lim_{ \eta ,\gamma \to 0} e^{-i k_{\eta,\gamma} x}  = \begin{cases}  e^{-i k^\e x} &  \text{if $(\eps,\mu) = (\eps^\e,\mu^\e)$, }  \\ e^{i k^\I x} \ & \text{if $(\eps,\mu) = (\eps^\I,\mu^\I)$. }  \end{cases}
\end{equation}
Classically, $e^{-i k^\e x}$ verifies the Sommerfeld radiation condition \eqref{eq:sommerfeld} but $ e^{i k^\I x}$ does not. Nevertheless this last quantity satisfies the ``reversed'' condition \eqref{eq:sommerfeldInverse}. This justifies the radiation conditions used in \eqref{eq:pbguide:nd2}.

\begin{rmk}\label{rmk:absorptionSign}
The choice of the sign for the imaginary part of $\eps_\eta$ and $\mu_\gamma$ is linked to the time convention $e^{-i \omega t}$. Indeed, under reasonable physical assumptions (passivity and causality) and with this convention, it is possible to show that $\eps_\eta$ and $\mu_\gamma$ (as function of $\omega$) are necessarily Herglotz functions, \textit{i.e.} analytical functions of the upper half-plane with positive imaginary parts (see for instance \cite{nussenzveig1972causality,vinoles2016problemes}).
\end{rmk}

\section{Appendix: Bessel and Hankel functions}
\label{sec:bessel}

Recall (see \textit{e.g.} \cite{watson1995treatise,olver2010nist}) that the Bessel functions are defined as the solutions of the ODE
\begin{equation}\label{eq:eqBessel}
r^2 \frac{ \mathrm d^2 y}{\mathrm d r^2} + r \frac{ \mathrm d y}{\mathrm d  r} +(r^2-\nu^2) y = 0,
\end{equation}
where $\nu\in \C$ is a parameter (in our case an integer or half an integer). Equation \eqref{eq:eqBessel} admits two linearly independent solutions $J_\nu$ (Bessel function of the first kind) and $Y_\nu$ (Bessel function of the second kind) defined by
\begin{equation}\label{eq:formuleJnu}
J_\nu(r) := \sum_{k = 0}^{+\infty} \frac{(-1)^k}{k! \,\Gamma(k+\nu+1)} \left( \frac r2\right)^{2k+\nu},\qquad r \geq 0
\end{equation}
where $\Gamma$ is the Gamma function and by
\begin{equation}
Y_\nu(r) :=\frac{J_\nu(r) \cos(\nu \pi) - J_{-\nu}(r)}{\sin(\nu \pi)},\qquad r>0.
\end{equation}
This last expression has to be understood as the limit value when $\nu = n \in \Z$: $Y_n = \lim_{\nu \to n} Y_\nu$. The spherical Bessel functions $j_\nu$ and $y_\nu$ are defined using the Bessel functions:
\begin{equation}
j_\nu(r) = \sqrt{\frac{\pi}{2r}} J_{\nu+1/2}(r) \qquad \text{and} \qquad y_\nu(r) = \sqrt{\frac{\pi}{2r}} Y_{\nu+1/2}(r).
\end{equation}
We also define the Hankel (reps.spherical Hanekl) function of the first kind $H_n := J_n + i Y_n$ (resp. $h_\ell = j_\ell + i y_\ell$). 

The linear independence of $J_\nu$ and $Y_\nu$ can be specified through the Wronskian formula: for all $\nu \in \C$ et $r>0$, one has (the derivatives are w.r.t. $r$)
\begin{equation} \label{eq:wronskien}
J_\nu(r)Y_\nu^\prime(r)-J^\prime_\nu(r)Y_\nu(r) = \frac{2}{\pi r}.
\end{equation}

Recall that we first want to prove Lemma \ref{lemme:determinant}. Actually we can prove the more general result: 
\begin{lemme}
Let $\alpha,\beta>0$ and $\lambda \in \R^*$. For any $\nu>0$ such that $\alpha$ is not a zero of $J_\nu$ and $\beta$ is not a zero of $H_\nu$, one has 
\begin{equation}\label{eq:besselTMP}
\frac{  J^\prime_\nu ( \alpha )}{J_\nu ( \alpha)}  + \lambda    \frac{  H^\prime_\nu( \beta )}{ H_\nu  (\beta ) }  \neq 0.
\end{equation}
\end{lemme}

\begin{proof}
Let $\nu > 0$ fixed. By contradiction, suppose that there exist $\alpha,\beta>0$ and $\lambda \in \R^*$ such that the left hand-side of \eqref{eq:besselTMP} is zero.
Taking its imaginary part and using the fact that $H_\nu = J_\nu + i Y_\nu$ one gets
\begin{equation}
\frac{J_\nu(\beta)Y_\nu^\prime(\beta)-J^\prime_\nu(\beta)Y_\nu(\beta) }{J_\nu(\beta)^2+Y_\nu(\beta)^2} = 0.
\end{equation}
This contradicts the Wronskian formula \eqref{eq:wronskien}.
\end{proof}

Now we want to prove the asymptotics \eqref{eq:asymptobessel} and \eqref{eq:asymptobessel3d}. We start with a lemma concerning $J_\nu$.
\begin{lemme}
Assume $\nu \in \R^*$ and $r>0$. Then 
\begin{equation}\label{eq:asymptoticBesselNu}
J_\nu(r) = \frac{r^\nu}{2^\nu \Gamma(\nu+1) } \left[ \sum_{k=0}^N  \frac{ (-1)^k \Gamma(\nu+1)}{k! \,  \Gamma(k+\nu+1)} \left( \frac{r}{2} \right)^{2k} + \mathcal O\left( \frac{1}{(\nu+1)^{N+1}} \right)  \right].
\end{equation}
\end{lemme}

\begin{proof}
From \eqref{eq:formuleJnu}, one gets
\begin{equation}\label{eq:formuleJnu2}
J_\nu(r) = \frac{r^\nu}{2^\nu \Gamma(\nu+1) } \left[ \sum_{k=0}^N  \frac{ (-1)^k \Gamma(\nu+1)}{k!\,  \Gamma(k+\nu+1)} \left( \frac{r}{2} \right)^{2k} + \sum_{k \geq N+1} \frac{ (-1)^k \Gamma(\nu)}{k!\,  \Gamma(k+\nu+1)} \left( \frac{r}{2} \right)^{2k}  \right].
\end{equation}
Since $\Gamma(k+\nu+1)/\Gamma(\nu+1)\geq (\nu+1)^{N+1}$ for $k \geq N+1$, by denoting $R_N$ the second sum of \eqref{eq:formuleJnu2}, one has
\begin{equation}
\vert R_N \vert \leq \frac{1}{(\nu+1)^{N+1}}  \sum_{k \geq N+1} \frac{1}{k!}\left( \frac{r}{2} \right)^{2k}. 
\end{equation}
Observing that this series is convergent, one gets \eqref{eq:asymptoticBesselNu}.
\end{proof}

We can already deduce some results from this lemma. For the asymptotic of $J_n(r)$ in \eqref{eq:asymptobessel}, one just need to take $\nu = n \in \N^*$ in \eqref{eq:asymptoticBesselNu} (since $\Gamma(n+1) = n!$). For $j_\ell(r)$, recall that $j_\ell(r) = \sqrt{\pi/(2r)}  J_{\ell+1/2}(r)$, so taking $\nu =  \ell + 1/2$, $\ell \in \N$ in \eqref{eq:asymptoticBesselNu} and
using 
\begin{equation}
\Gamma\left(\ell + \frac{1}{2}\right) = \sqrt{\pi} \frac{(2\ell-1)!!}{2^\ell},\qquad \ell \in \N,
\end{equation}
gives the asymptotic of $j_\ell(r)$ in \eqref{eq:asymptobessel3d}. The asymptotics for $J_n^\prime(r)$ and $j_\ell^\prime(r)$ are deduced easily from the ones of $J_n(r)$ and $j_\ell(r)$. Concerning the Hankel functions $H_n$ and $h_\ell$, we first need the ones for $Y_n$ and $y_\ell$. For the last, it is straightforward: using $y_\ell(r) = (-1)^{\ell+1} \sqrt{\pi/(2 r)} J_{-(\ell+1/2)}(r)$ and taking $\nu =  -(\ell + 1/2)$ in \eqref{eq:asymptoticBesselNu} lead to 
\begin{equation}\label{eq:tmpBessel}
y_\ell (r)  = -\frac{(2\ell-1)!!}{r^{\ell+1}} \left[ \sum_{k=0}^N \frac{ (2\ell-2k-1)!!}{k! } \frac{1}{(2\ell-1)!!}  \left( \frac{r^2}{2} \right)^k+ \mathcal O\left(\frac{1}{\ell^{N+1}} \right)\right]. 
\end{equation}
To deduce the result for $h_\ell$ in \eqref{eq:asymptobessel3d}, it suffices to notice that $j_\ell$ is negligible compared to $y_\ell$, so $h_\ell(r) \sim iy_\ell(r)$ when $\ell$ tends to $+\infty$ so the asymptotic of $h_\ell(r)$ in \eqref{eq:asymptobessel3d} is directly given by \eqref{eq:tmpBessel} and the ones for $h_\ell^\prime(r)$ are deduced easily from them. 

For the asymptotic of $Y_n$, we cannot do it directly. We have to use that
\begin{equation}\label{eq:eqbessel1}
Y_n(r)  = \frac{2}{\pi}\left[ \log\left(\frac{r}{2}\right) + \gamma \right] - \frac{1}{\pi} \sum_{k=0}^{n-1} \frac{(n-1-k)!}{k!} \left( \frac{2}{r}\right)^{n-2k} 
-\frac{1}{\pi} \sum_{k=0}^{+\infty} \frac{(-1)^k}{k!(n+k)!} \left( \frac{r}{2}\right)^{2k+n} \left[ \psi(k+n)+\psi(k) \right],
\end{equation}
where
\begin{equation}
\psi(k) := \sum_{m=1}^k \frac{1}{m} \qquad \text{and} \qquad \gamma:=\lim_{k \to +\infty} \left( \psi(k) - \log(k) \right)\approx 0,5772\cdots,
\end{equation}
are respectively the partial sums of the harmonic series and the Euler-Mascheroni constant. First notice that the first term of \eqref{eq:eqbessel1} does not depend on $n$, so $2[\log(r/2)-\gamma]/\pi = \mathcal O(1)$. The third term is bounded with respect to $n$ too, because $\psi(n+k) = \log(n+k) + \gamma + \mathcal O(1/n)$ (see for instance \cite{conway2012book}) and $r^n \log(k+n) / (n+k)!$ is bounded. Thus we get from \eqref{eq:eqbessel1}
\begin{equation}
Y_n(r) = \frac{1}{ \pi} \frac{2^n (n-1)!}{r^n} \left[\sum_{k = 0}^N \frac{ (n-k-1)!}{k!(n-1)!}  \left( \frac{r}{2} \right)^{2k}+\mathcal O\left(\frac{1}{n^{N+1}} \right)\right] .
\end{equation}
To deduce the result for $H_n$ in \eqref{eq:asymptobessel}, it suffices to notice that $J_n$ is negligible compared to $Y_n$, so the asymptotic of  $H_n$ in \eqref{eq:asymptobessel3d} is directly given by \eqref{eq:tmpBessel}. The ones for $H_n^\prime(r)$ are then deduced easily.

\section*{Acknowledgment}

The author deeply thanks Camille Carvalho and Lucas Chesnel for their useful comments and suggestions on this work.

\bibliographystyle{imamat}
\bibliography{biblio}

\begin{thebibliography}{}

\bibitem[Bonnet-Ben~Dhia et~al., 2012]{dhia2012t}
Bonnet-Ben~Dhia, A.-S., Chesnel, L. {\&} Ciarlet~Jr, P. (2012)  T-coercivity
  for scalar interface problems between dielectrics and metamaterials. {\em
  ESAIM: Mathematical Modelling and Numerical Analysis}, \textbf{46}(06),
  1363--1387.

\bibitem[Bonnet-Ben~Dhia et~al., 2014]{dhia2014t}
Bonnet-Ben~Dhia, A.-S., Chesnel, L. {\&} Ciarlet~Jr, P. (2014)  T-coercivity
  for the Maxwell problem with sign-changing coefficients. {\em Communications
  in Partial Differential Equations}, \textbf{39}(6), 1007--1031.

\bibitem[Bonnet-Ben~Dhia et~al., 2013]{bonnet2013radiation}
Bonnet-Ben~Dhia, A.-S., Chesnel, L. {\&} Claeys, X. (2013)  Radiation condition
  for a non-smooth interface between a dielectric and a metamaterial. {\em
  Mathematical Models and Methods in Applied Sciences}, \textbf{23}(09),
  1629--1662.

\bibitem[Bonnet-Ben~Dhia et~al., 2009]{dhia2009diffraction}
Bonnet-Ben~Dhia, A.-S., Dakhia, G., Hazard, C. {\&} Chorfi, L. (2009)
  Diffraction by a defect in an open waveguide: a mathematical analysis based
  on a modal radiation condition. {\em SIAM Journal on Applied Mathematics},
  \textbf{70}(3), 677--693.

\bibitem[Bouchitt{\'e} \& Schweizer, 2010a]{bouchitte2010cloaking}
Bouchitt{\'e}, G. {\&} Schweizer, B. (2010a)  Cloaking of small objects by
  anomalous localized resonance. {\em The Quarterly Journal of Mechanics and
  Applied Mathematics}, \textbf{63}(4), 437--463.

\bibitem[Bouchitt{\'e} \& Schweizer, 2010b]{bouchitte2010homogenization}
Bouchitt{\'e}, G. {\&} Schweizer, B. (2010b)  Homogenization of Maxwell's
  equations in a split ring geometry. {\em Multiscale Modeling \& Simulation},
  \textbf{8}(3), 717--750.

\bibitem[Cakoni \& Colton, 2005]{cakoni2005qualitative}
Cakoni, F. {\&} Colton, D. (2005) {\em Qualitative methods in inverse
  scattering theory: An introduction}.
Springer Science \& Business Media.

\bibitem[Carvalho, 2015]{carvalho2015etude}
Carvalho, C. (2015) {\em {\'E}tude math{\'e}matique et num{\'e}rique de
  structures plasmoniques avec coins}.
PhD thesis, \'Ecole Polytechnique.

\bibitem[Cassier, 2014]{cassier2014etude}
Cassier, M. (2014) {\em {\'E}tude de deux probl{\`e}mes de propagation d'ondes
  transitoires: 1) Focalisation spatio-temporelle en acoustique; 2)
  Transmission entre un di{\'e}lectrique et un m{\'e}tamat{\'e}riau}.
PhD thesis, \'Ecole Polytechnique.

\bibitem[Colton \& Kress, 2012]{colton2012inverse}
Colton, D. {\&} Kress, R. (2012) {\em Inverse acoustic and electromagnetic
  scattering theory}, volume~93.
Springer Science \& Business Media.

\bibitem[Conway \& Guy, 2012]{conway2012book}
Conway, J.~H. {\&} Guy, R. (2012) {\em The book of numbers}.
Springer Science \& Business Media.

\bibitem[Costabel \& Stephan, 1985]{costabel1985direct}
Costabel, M. {\&} Stephan, E. (1985)  A direct boundary integral equation
  method for transmission problems. {\em Journal of mathematical analysis and
  applications}, \textbf{106}(2), 367--413.

\bibitem[Cui et~al., 2010]{cui2010metamaterials}
Cui, T.~J., Smith, D.~R. {\&} Liu, R. (2010) {\em Metamaterials: theory,
  design, and applications}.
Springer.

\bibitem[Davies, 1996]{davies1996spectral}
Davies, E.~B. (1996) {\em Spectral theory and differential operators},
  volume~42.
Cambridge University Press.

\bibitem[{\`E}idus \& Hill, 1963]{eidus1963principle}
{\`E}idus, D.~M. {\&} Hill, C.~D. (1963)  On the principle of limiting
  absorption. Technical report, DTIC Document.

\bibitem[Gralak \& Maystre, 2012]{gralak2012negative}
Gralak, B. {\&} Maystre, D. (2012)  Negative index materials and time-harmonic
  electromagnetic field. {\em Comptes Rendus Physique}, \textbf{13}(8),
  786--799.

\bibitem[Hazard \& Lun{\'e}ville, 2008]{hazard2008improved}
Hazard, C. {\&} Lun{\'e}ville, E. (2008)  An improved multimodal approach for
  non-uniform acoustic waveguides. {\em IMA journal of applied mathematics},
  \textbf{73}(4), 668--690.

\bibitem[Huet, 1976]{huet1976decomposition}
Huet, D. (1976) {\em D{\'e}composition spectrale et op{\'e}rateurs}, volume~16.
Presses universitaires de France.

\bibitem[Iorio~Jr \& de~Magalh{\~a}es~Iorio, 2001]{iorio2001fourier}
Iorio~Jr, R.~J. {\&} de~Magalh{\~a}es~Iorio, V. (2001) {\em Fourier analysis
  and partial differential equations}, volume~70.
Cambridge University Press.

\bibitem[Lamacz \& Schweizer, 2013]{lamacz2013effective}
Lamacz, A. {\&} Schweizer, B. (2013)  Effective Maxwell equations in a geometry
  with flat rings of arbitrary shape. {\em SIAM Journal on Mathematical
  Analysis}, \textbf{45}(3), 1460--1494.

\bibitem[Li, 2016]{li2016literature}
Li, J. (2016)  A litterature survey of mathematical study of metamaterials.
  {\em International journal of numerical analysis and modeling},
  \textbf{13}(2), 230--243.

\bibitem[Linton \& McIver, 2007]{linton2007embedded}
Linton, C.~M. {\&} McIver, P. (2007)  Embedded trapped modes in water waves and
  acoustics. {\em Wave motion}, \textbf{45}(1), 16--29.

\bibitem[Lions \& Magenes, 2012]{lions2012non}
Lions, J.~L. {\&} Magenes, E. (2012) {\em Non-homogeneous boundary value
  problems and applications}, volume~1.
Springer Science \& Business Media.

\bibitem[Maier, 2007]{maier2007plasmonics}
Maier, S.~A. (2007) {\em Plasmonics: fundamentals and applications}.
Springer Science \& Business Media.

\bibitem[Malyuzhinets, 1951]{malyuzhinets1951note}
Malyuzhinets, G.~D. (1951)  A note on the radiation principle. {\em Zhurnal
  technicheskoi fiziki}, \textbf{21}(8), 940--942.

\bibitem[Milton \& Nicorovici, 2006]{milton2006cloaking}
Milton, G.~W. {\&} Nicorovici, N.-A.~P. (2006)  On the cloaking effects
  associated with anomalous localized resonance. In {\em Proceedings of the
  Royal Society of London A: Mathematical, Physical and Engineering Sciences},
  volume 462, pages 3027--3059. The Royal Society.

\bibitem[Morse \& Feshbach, 1953]{morse1953methods}
Morse, P.~M. {\&} Feshbach, H. (1953)  Methods of theoretical physics. {\em
  International Series in Pure and Applied Physics, New York: McGraw-Hill,
  1953}, \textbf{1}.

\bibitem[Nguyen, 2015]{nguyen2015asymptotic}
Nguyen, H.-M. (2015)  Asymptotic behavior of solutions to the Helmholtz
  equations with sign changing coefficients. {\em Transactions of the American
  Mathematical Society}, \textbf{367}(9), 6581--6595.

\bibitem[Nguyen, 2016]{nguyen2016limiting}
Nguyen, H.-M. (2016)  Limiting absorption principle and well-posedness for the
  Helmholtz equation with sign changing coefficients. {\em Journal de
  Math{\'e}matiques Pures et Appliqu{\'e}es}.

\bibitem[Nussenzveig, 1972]{nussenzveig1972causality}
Nussenzveig, H.~M. (1972)  Causality and dispersion relations. .

\bibitem[Ola, 1995]{ola1995remarks}
Ola, P. (1995)  Remarks on a transmission problem. {\em Journal of mathematical
  analysis and applications}, \textbf{196}(2), 639--658.

\bibitem[Olver, 2010]{olver2010nist}
Olver, F. W.~J. (2010) {\em NIST handbook of mathematical functions}.
Cambridge University Press.

\bibitem[Pendry, 2004]{pendry2004negative}
Pendry, J.~B. (2004)  Negative refraction. {\em Contemporary Physics},
  \textbf{45}(3), 191--202.

\bibitem[Smith et~al., 2004]{smith2004metamaterials}
Smith, D.~R., Pendry, J.~B. {\&} Wiltshire, M. C.~K. (2004)  Metamaterials and
  negative refractive index. {\em Science}, \textbf{305}(5685), 788--792.

\bibitem[Stein \& Weiss, 1971]{stein1971introduction}
Stein, E.~M. {\&} Weiss, G.~L. (1971) {\em Introduction to Fourier analysis on
  Euclidean spaces}, volume~1.
Princeton university press.

\bibitem[Taflove \& Hagness, 2005]{taflove2005computational}
Taflove, A. {\&} Hagness, S.~C. (2005) {\em Computational electrodynamics}.
Artech house.

\bibitem[Vinoles, 2016]{vinoles2016problemes}
Vinoles, V. (2016) {\em Probl{\`e}mes d'interface en pr{\'e}sence de
  m{\'e}tamat{\'e}riaux: mod{\'e}lisation, analyse et simulations}.
PhD thesis, Paris-Saclay University.

\bibitem[Watson, 1995]{watson1995treatise}
Watson, G.~N. (1995) {\em A treatise on the theory of Bessel functions}.
Cambridge university press.

\bibitem[Weder, 2012]{weder2012spectral}
Weder, R. (2012) {\em Spectral and scattering theory for wave propagation in
  perturbed stratified media}, volume~87.
Springer Science \& Business Media.

\bibitem[Ziolkowski \& Heyman, 2001]{ziolkowski2001wave}
Ziolkowski, R.~W. {\&} Heyman, E. (2001)  Wave propagation in media having
  negative permittivity and permeability. {\em Physical review E},
  \textbf{64}(5), 056625.

\end{thebibliography}

\end{document}